\documentclass[11pt]{amsart}

\usepackage{geometry}
\usepackage{graphicx}
\usepackage{amsmath,amssymb,amsthm}
\usepackage{enumerate}
\usepackage{hyperref}
\usepackage[active]{srcltx}
\usepackage[T1]{fontenc}
\usepackage{color}
\newtheorem{theo}{Theorem}
  
\newtheorem{lem}{Lemma}[section]
\newtheorem{defi}{Definition}[section]
\newtheorem{cor}{Corollary}[section]
\newtheorem{prop}{Proposition}[section]
\newtheorem{rmk}{Remark}[section]
\newcommand{\eps}{\varepsilon}

\newcommand{\M}{\mathcal{M}}
\newcommand{\MK}{\mathtt{MK}}
\renewcommand{\v}{\mathbf{v}}
\newcommand{\rd}{\mathrm d}
\newcommand{\R}{\mathbb{R}}
\newcommand{\W}{\MK}
\renewcommand{\H}{\mathtt{FR}}
\renewcommand{\d}{\mathtt{KFR}}
\newcommand{\KFR}{\mathtt{KFR}}
\newcommand{\FR}{\mathtt{FR}}

\newcommand{\N}{\mathbb{N}}
\newcommand{\tg}[1]{{ #1}}

\numberwithin{equation}{section}

\def\dive{\operatorname{div}}
\def\grad{\operatorname{grad}}
\begin{document}
\title[A JKO splitting scheme for $\KFR$ gradient flows]{A JKO splitting scheme for Kantorovich-Fisher-Rao gradient flows}
\date{}
\author{Thomas O. Gallou\"et}
\address{Thomas O. Gallou\"et: D\'epartement de math\'ematiques, Universit\'e de Li\`ege, All\'ee de la d\'ecouverte 12, B-4000 Li\`ege, Belgique.
E-mail: {\tt
 thomas.gallouet@ulg.ac.be}}
\author{L\'eonard Monsaingeon}
\address{L\'eonard Monsaingeon: Institut \'Elie Cartan de Lorraine, Universit\'e de Lorraine, Site de Nancy, B.P. 70239, F-54506 Vandoeuvre-l\`es-Nancy Cedex, France. E-mail: {\tt  leonard.monsaingeon@univ-lorraine.fr
}}


\maketitle

\begin{abstract}
In this article we set up a splitting variant of the Jordan-Kinderlehrer-Otto scheme in order to handle gradient flows with respect to the Kantorovich-Fisher-Rao metric, recently introduced and defined on the space of positive Radon measure with varying masses. 
We perform successively a time step for the quadratic Wasserstein/Monge-Kantorovich distance, and then for the Hellinger/Fisher-Rao distance.
Exploiting some inf-convolution structure of the metric we show convergence of the whole process for the standard class of energy functionals under suitable compactness assumptions, and investigate in details the case of internal energies. 
The interest is double: On the one hand we prove existence of weak solutions for a certain class of reaction-advection-diffusion equations, and on the other hand this process is constructive and well adapted to available numerical solvers. 
\end{abstract}
\section{Introduction}

A new Optimal Transport distance on the space of positive Radon measures has been recently introduced independently by three different teams \cite{Peyre_unbalanced_15,Peyre_inter_15,KMV_15,LMS_big_15,LMS_small_15}.
Contrarily to the classical Wasserstein-Monge-Kantorovich distances, which are restricted to the space of measures with fixed mass (typically probability measures), this new distance has the advantage of allowing for mass variations, can be computed between arbitrary measures, and does not require decay at infinity (such as finite moments).
In \cite{Peyre_unbalanced_15,Peyre_inter_15} the distance is called Wasserstein-Fisher-Rao and is introduced with imaging applications in mind.
In \cite{LMS_big_15,LMS_small_15} the distance is referred to as the Hellinger-Kantorovich one, and was studied as a particular case of a larger class of Optimal Transport problems including primal/dual and static formulations.
The second author introduced the same distance in \cite{KMV_15}, with applications to population dynamics and gradient flows in mind.
In this paper we propose the name Kantorovich-Fisher-Rao for this metric ($\KFR$ in the sequel), taking into account all contributions.

On one side we aim here at understanding the local behavior of the $\KFR$ metric with respect to the by now classical  quadratic Monge-Kantorovich/Wasserstein metric $\MK$ and the Hellinger/Fisher-Rao metric $\FR$.
On the other side we want to use this information to prove existence of weak solutions to gradient flows while avoiding to look too closely into the geometry of the $\KFR$ space.
Moreover our constructive approach is naturally adapted to available numerical schemes and Monge-Amp\`ere solvers. 

A possible way to formalize abstract  gradient flow structures is to prove convergence of the corresponding Minimizing Movement scheme, as introduced by De Giorgi \cite{DeGio_93} and later exploited by Jordan-Kinderlehrer-Otto for the $\MK$ metric \cite{JKO_98}.
Given a metric space $(X,d)$ and a functional $F:X\to\R$, the JKO scheme with time-step $\tau>0$ writes 
\begin{equation}
x^{n+1}\in \underset{x\in X}{\operatorname{Argmin}}\left\{\frac{1}{2\tau}d^2 (x,x^n)+ F(x)\right\}.
\label{eq:minimizing_scheme_intro}
\end{equation}
Letting $\tau\to 0$ one should expect to recover a weak solution of the gradient flow 
\begin{equation}
\dot x (t) =-\grad_d F(x(t)).
\label{eq:grad_flow_intro}
\end{equation}
Looking at \eqref{eq:grad_flow_intro}, which is a differential equality between infinitesimal variations, we guess that only the local behavior of the metric $d$ matters in \eqref{eq:minimizing_scheme_intro}. 

The starting point of our analysis is therefore the local structure of the Kantorovich-Fisher-Rao metric, which endows the space of positive Radon measures $\rho\in \M^+$ with a formal Riemannian structure \cite{KMV_15}.
Based on some inf-convolution structure, our heuristic considerations will suggest that, infinitesimally, $\KFR$ should be the orthogonal sum of $\MK$ and $\FR$: 
\begin{equation*}
\d^2\approx \W^2+\H^2.
\end{equation*}
More precisely, we will show that in the tangent plane there holds
\begin{equation}
\label{diff_intro}
\|\grad _\d\mathcal F(\rho)\|^2
 = \|\grad _\W\mathcal F(\rho)\|^2
+\|\grad _\H\mathcal F(\rho)\|^2
\end{equation}
at least formally for reasonable functionals $\mathcal F$, and this is in fact the key point in this work.
The notion of metric gradients and tangent norms appearing in \eqref{diff_intro} will be precised in section~\ref{sec:prel}.
This naturally leads to a splitting approach for $\KFR$ Minimizing Movements: we successively run a first time step for $\MK$, leading to the diffusion term in the associated PDE, and then a second step for $\FR$, leading to the reaction term in the PDE.
This can also be viewed as replacing the direct approximation ``by hypotenuses'' in the JKO scheme (with the $\KFR$ distance) by a double approximation ``by legs'' (each of the legs corresponding to one of the $\FR,\MK$ metrics).
Formula \eqref{diff_intro} also indicates that the energy dissipation $D(t):=-\frac {dF}{dt}=|\dot x|^2=|\grad F|^2$ will be correctly approximated in \eqref{eq:grad_flow_intro}. 
One elementary Monge-Kantorovich JKO step is now well known, see for instance \cite{Santambroggio_book} and references therein.
On the other hand the Fisher-Rao metric enjoys a Riemannian structure that can be recast, up to a change of variable, into a convex Hilbertian setting, and therefore the reaction step should be easy to handle numerically.

Here we show that the classical estimates (energy monotonicity, total square distance, mass control, BV\ldots) propagate along each $\MK$ and $\FR$ substeps, and nicely fit together in the unified $\KFR$ framework.
This allows us to prove existence of weak solutions for a whole class of reaction-advection-diffusion PDEs
$$
\partial_t\rho =\dive(\rho\nabla(U'(\rho)+\Psi+K\ast\rho))- \rho (U'(\rho)+\Psi+K\ast\rho)
$$
associated with $\KFR$ gradient flows
$$
\partial_t\rho =-\grad_{\KFR}\mathcal F(\rho),\qquad \mathcal F(\rho)=\int_{\Omega}\big\{U(\rho)+\Psi(x)\rho+\frac 12 \rho K\star \rho\big\}.
$$
The structural conditions on the \emph{internal energy} $U$, \emph{external potential} $\Psi$, \emph{interaction kernel} $K$, and the meaning of the \emph{metric gradient} $\grad_{\KFR}$ will be precised later on.
Moreover we retrieve a natural Energy Dissipation Inequality at least in some particular cases, which is well known \cite{AGS_08} to completely characterize metric gradient flows.

Our splitting method has several interests:
First we avoid a possibly delicate geometrical analysis of the $\KFR$ space, in particular we do not need to differentiate the squared $\KFR$ distance. 
This is usually required to derive the Euler-Lagrange equations in the JKO scheme, but might not be straightforward here (see Section~\ref{section:uncoupling_inf-sup_convolution} for discussions).
Secondly, the approach leads to a new constructive existence proof for weak solutions to the above class of PDEs, and can be implemented numerically (see \cite{kinderlehrer1999approximation} for an early application of this idea).
For one elementary $\MK$ step many discretizations are now available, such as the semi-discrete scheme \cite{merigot2011multiscale,benamou2014discretization}, the augmented Lagrangian procedure \cite{benamou_AL2JKO}, or the Entropic relaxation \cite{peyre2015entropic}. 
The Fisher-Rao minimizing step should not be difficult to implement, since the problem is convex with the good choice of variables.

Finally it is worth stressing that the $\KFR$ distance is, by construction, well adapted to handle general transport and reaction processes in a unified framework.
One very natural extension of this work would be to consider two separate energy functionals $\mathcal F_1,\mathcal F_2$, to be used respectively in the diffusion and reaction parts. 
This natural approach is the purpose of our ongoing works \cite{GLM_16,theseMax} and should allow to treat more general equations (not necessarily gradient flows).
However, the rigorous analysis requires suitable compatibility conditions between the two driving functionals and becomes quite technical (see e.g. Remark~\ref{rmk:total_square_dist_2_functionals}).
For the sake of exposition we chose to restrict here to the case of pure gradient flows $\mathcal F_1=\mathcal F=\mathcal F_2$, when the technical estimates are more straightforward and allow to recover dissipation estimates (see Section~\ref{section:EDI}).\\

The paper is structured as follows.
In Section \ref{sec:prel} we recall some basic facts on the three metrics involved: the quadratic Monge-Kantorovich $\MK$, the Fisher-Rao $\FR$, and the Kantorovich-Fisher-Rao $\KFR$ distances.
We highlight the three differential Riemannian structures and gradient flow interpretations.
Section \ref{section:uncoupling_inf-sup_convolution} details the local relation between the three metrics, in particular the infinitesimal uncoupling of the inf-convolution.
For the sake of exposition we deliberately remain formal in order to motivate the rigorous analysis in the next sections.
In section \ref{section:minimizing_scheme} we define the splitting minimizing movement scheme for the $\KFR$ distance and prove, under natural compactness assumptions, the convergence towards a weak solution of the expected PDE.
As an example in section \ref{section:compactness} we work out all the technical details for the particular case of internal energies, and show that the previous abstract compactness hypothesis holds.

\section{Preliminaries}\label{sec:prel}
From now on we always assume that $\Omega\subset \R^d$ is a convex subset, possibly unbounded.
In this section we recall some facts about the Wasserstein-Monge-Kantorovich and Hellinger-Fisher-Rao distances $\MK,\H$, and introduce the Kantorovich-Fisher-Rao distance $\KFR$.
We also present the differential points of view for each of them, allowing to retrieve the three corresponding pseudo Riemannian structures and compute gradients of functionals with respect to the $\W,\H,\d$ metrics.
\subsection{The quadratic Monge-Kantorovich distance $\W$}
\label{section:Wasserstein_distance}
We refer to \cite{villani_small} for an introduction and to \cite{villani_big} for a complete overview of the Wasserstein-Monge-Kantorovich distances.
\begin{defi}
\label{defi:W_plans}
For any nonnegative Radon measures $\rho_0,\rho_1\in \M^+_2$ with same mass $|\rho_0|=m=|\rho_1|$ and finite second moments, the quadratic Monge-Kantorovich distance is
\begin{equation}
\label{eq:def_W_plans}
\W^2(\rho_0,\rho_1)=\min\limits_{\gamma\in\Gamma[\rho_0,\rho_1]}\int_{\Omega\times\Omega}|x-y|^2\rd\gamma(x,y),
\end{equation}
where the admissible set of \emph{transference plans} $\Gamma[\rho_0,\rho_1]$ consists of nonnegative measures $\gamma\in \M^+(\Omega\times\Omega)$ with mass $|\gamma|=m$ and prescribed marginals $\Pi_x(\gamma)=\rho_0(x)$ and $\Pi_y(\gamma)=\rho_1(y)$.
\end{defi}
\noindent
The minimizer is unique and is called an optimal plan.
When $\rho_0$ does not charge small sets we have the characterization in terms of transport \emph{maps}:
\begin{theo}[Brenier, Gangbo-McCann, \cite{brenier1991polar,gangbo_mccann_96}]
 \label{theo:Gangbo_McCann}
 With the same assumptions as in Definition~\ref{defi:W_plans}, assume that $\rho_0$ does not give mass to $\mathcal H^{d-1}$ sets. Then
 \begin{equation}
  \label{eq:def_W_maps}
 \W^2(\rho_0,\rho_1)=\min\limits_{\rho_1=\mathbf t \#\rho_0}\int_\Omega|x-\mathbf t(x)|^2\rd\rho_0(x),
 \end{equation}
 and the optimal transport map $\mathbf t$ is unique $\rd\rho_0$ almost everywhere. 
\end{theo}
\noindent
We recall the definition of pushforwards by maps $\mathbf t:\Omega\to\Omega$
$$
\rho_1=\mathbf t\# \rho_0
\quad\Leftrightarrow\quad
\int_\Omega \phi(y)\rd\rho_1(y)=\int_\Omega\phi(\mathbf t(x))\rd\rho_0(x)
\quad\mbox{for all }\phi\in\mathcal C_c(\Omega).
$$
\noindent
As first pointed out by Benamou and Brenier \cite{BB_00} we also have the following dynamic representation of the Wasserstein distance:
\begin{theo}[Benamou-Brenier formula, \cite{AGS_08,BB_00}]
\label{theo:benamou_brenier}
There holds
\begin{equation}
\label{eq:benamou_brenier_formula}
\W^2(\rho_0,\rho_1)=\min\limits_{(\rho,\v)\in\mathcal A_\W[\rho_0,\rho_1]}\int_0^1\int_\Omega|\v_t|^2\rd\rho_t\rd t,
\end{equation}
where the admissible set $\mathcal A_\W[\rho_0,\rho_1]$ consists of curves $[0,1]\ni t\mapsto (\rho_t,\v_t)\in \mathcal  M^+(\Omega)\times L^2(\Omega,\rd\rho_t)^d$ such that $t\mapsto\rho_t$ is narrowly continuous with endpoints $\rho_0,\rho_1$ and solving the continuity equation
$$
\partial_t\rho_t+\dive(\rho_t\v_t)=0
$$
in the sense of distributions $\mathcal D'((0,1)\times\Omega)$.
\end{theo}
\begin{rmk}
\label{rmk:functional_space_L2(drho)}
 Note that, since we are minimizing the kinetic energy in \eqref{eq:benamou_brenier_formula}, the admissible velocity fields $\v$ are implicitly taken in the varying weighted space $\v\in L^2(0,1;L^2(\rd\rho_t))$.
 For such velocities in this energy space, the action of the product $\rho_t\v_t$ is well defined against any smooth test-function $\varphi\in \mathcal C^\infty_c((0,1)\times\Omega)\subset L^2(0,1;L^2(\rd\rho_t))$ in the distributional formulation of the continuity equation, i-e
 $$
 -\langle \dive(\rho \v),\varphi\rangle_{\mathcal D',\mathcal D}=\langle \rho \v,\nabla\varphi\rangle_{\mathcal D',\mathcal D}=\left( \v,\nabla\varphi\right)_{L^2(0,1;L^2(\rd\rho_t))}=\int_0^1\int_{\Omega} \v_t\cdot \nabla \varphi \,\rd\rho_t\,\rd t.
 $$
\end{rmk}
\noindent
In \eqref{eq:benamou_brenier_formula} a minimizing curve $t\mapsto\rho_t$ is of course a geodesics, with constant metric speed $\|\v_t\|^2_{L^2(\rd\rho_t)}=cst=\W^2(\rho_0,\rho_1)$.
Note that we allow here for any arbitrary mass $|\rho_0|=m=|\rho_1|>0$, and that the distance scales as $\W^2(\alpha\rho_0,\alpha\rho_1)=\alpha\W^2(\rho_0,\rho_1)$. 
This is apparent in all three formulations \eqref{eq:def_W_plans}\eqref{eq:def_W_maps}\eqref{eq:benamou_brenier_formula}, which are linear in $\gamma$, $\rho_0,\rho_1$, and $\rho_t$ respectively.

As is now well-known from the works of Otto \cite{Otto_PME_2001}, we can view the set of measures with fixed mass as a pseudo-Riemannian manifold, endowing the tangent plane
$$
T_\rho\M^+_\W=\{\partial_t\rho=-\dive(\rho\v)\qquad \mbox{evaluated at }t=0\}
$$
with the metrics
$$
\|\partial_t\rho\|^2_{T_\rho\M^+_\W}:=\inf\left\{\|\v\|^2_{L^2(\rd\rho)}:\quad \partial_t\rho=-\dive(\rho\v)\right\}.
$$

It is easy to see that, among all possible velocities $\v$ representing the same tangent vector $\partial_t\rho=-\dive(\rho\v)$, there is a unique one with minimal $L^2(\rd\rho)$ norm. 
A standard computation \cite{villani_small} shows that this particular velocity is necessarily potential, $\v=\nabla p$ for a pressure function $p$ uniquely defined up to constants (see the proof of Proposition~\ref{prop:coupling=uncoupling_d} below at least for smooth positive densities $\rho$). 
As a consequence we always choose to represent
$$
\|\partial_t\rho\|^2_{T_\rho\M^+_\W}=\|\nabla p\|^2_{L^2(\rd\rho)}
\quad \mbox{with the identification }\partial_t\rho=-\dive(\rho\nabla p).
$$
Here we remained formal and refer again to \cite{villani_small,villani_big} for details.
Now metric gradients $\grad_\W$ can be computed by the chain rule as follows: If $\partial_t\rho_t=-\dive(\rho_t \nabla p_t)$ is a smooth curve passing through $\rho_t(0)=\rho$ with arbitrary initial velocity $\zeta=\partial_t\rho(0)=-\dive(\rho \nabla p)$ then for functionals $\mathcal{F}(\rho)=\int_\Omega F(\rho(x),x)\rd x$
\begin{align*}
\left<\grad_\W \mathcal F(\rho),\zeta\right>_{T_\rho\M^+_\W} & =\left.\frac{d}{dt}\mathcal F (\rho_t)\right|_{t=0} = \left.\frac{d}{dt}\left(\int_\Omega F(\rho_t(x),x)\rd x\right)\right|_{t=0}\\
& =\int_{\Omega}F'(\rho)\times\{-\dive(\rho\nabla p)\}=\int_\Omega \nabla F'(\rho)\cdot \nabla p\,\rd\rho\\
& =\left( \nabla F'(\rho),\nabla p \right)_{L^2(\rd\rho)},
\end{align*}
where $F'(\rho)=\frac{\delta F}{\delta\rho}$ stands for the standard first variation with respect to $\rho$.
For the classical case $\mathcal F(\rho)=\int_\Omega \{U(\rho) + \Psi \rho +\frac 12 \rho K\star\rho \}$ considered here this means $F'(\rho)=U'(\rho)+\Psi(x)+K\star\rho$.
This shows that one should identify gradients
\begin{equation*}
\grad_\W\mathcal F(\rho)=-\dive(\rho \nabla F'(\rho))
\label{eq:formula_grad_W}
\end{equation*}
through the $L^2(\rd\rho)$ action in the tangent plane, and as a consequence the Monge-Kantorovich gradients flows read
\begin{equation}
\partial_t\rho=-\grad_\W\mathcal F(\rho)
\qquad \leftrightarrow\qquad \partial_t\rho=\dive(\rho F'(\rho)).
\label{eq:formula_grad_flow_W}
\end{equation}
%
\subsection{The Fisher-Rao distance $\H$}
\label{section:HFR_metrics}
The classical Hellinger-Kakutani distance \cite{hellinger,kakutani1948equivalence}, or Fisher-Rao metric, was first introduced for probability measures and is well known in statistics and information theory for its  connections with the Kullback's divergence and Fisher information \cite{bogachev_measure_theory}.
It can be extended to arbitrary nonnegative measures as
\begin{defi}
 The Fisher-Rao distance between measures $\rho_0,\rho_1\in\M^+$ is given by
 \begin{equation}
 \label{eq:def_HFR}
 \H^2(\rho_0,\rho_1)\overset{\text{def}}{=}\min\limits_{{ (\rho,r)}\in \mathcal A_{\H}[\rho_0,\rho_1]}\int_0^1\int_{\Omega}|r_t(x)|^2 \rd \rho_t(x)\,\rd t = 4\int_{\Omega}\left|\sqrt{\frac{\rd \rho_0}{\rd\lambda}}-\sqrt{\frac{\rd \rho_1}{\rd\lambda}}\right|^2\rd\lambda.
\end{equation}
 
 The admissible set $\mathcal A_{\H}[\rho_0,\rho_1]$ consists of curves $ [0,1]\ni t\mapsto (\rho_t,r_t)\in\mathcal M^+(\Omega)\times L^2(\Omega,\rd\rho_t)$ such that $ t\mapsto \rho_t$ is narrowly continuous with endpoints $ \rho_0,\rho_1$, and
 $$
 \partial_t\rho_t=\rho_tr_t
 $$
 in the sense of distributions $\mathcal D'((0,1)\times\Omega)$.
 \end{defi}
 { As in Remark~\ref{rmk:functional_space_L2(drho)}
the reaction term $r$ implicitly belongs to the energy space $L^2(0,1;L^2(\rd\rho_t))$, so that $\rho r$ is a well-defined distribution $\mathcal D'((0,1)\times\Omega)$ through the $\left(r,.\right)_{L^2(0,1;L^2(\rd\rho_t))}$ scalar product.
 }
  In the last explicit formula $\lambda$ is any reference measure such that $\rho_0,\rho_1$ are both absolutely continuous with respect to $\lambda$, with Radon-Nikodym derivatives $\frac{\rd\rho_i}{\rd\lambda}$.
 By $1$-homogeneity this expression doe not depend on the choice of $\lambda$, and the normalizing factor $4$ is chosen so that the metric  for the pivot space in the first dynamic formulation is exactly $L^2(\rd\rho_t)$ and not some other multiple $\beta{L^2(\rd\rho_t)}$.

At least for absolutely continuous measures $\rd\rho_0,\rd\rho_1\ll\rd x$ one can check that the minimum in the first definition is attained along the geodesic
$$
\rho_t=[(1-t)\sqrt{\rho_0}+t\sqrt{\rho_1}]^2
\quad\mbox{and}\quad
r_t:=2\frac{\sqrt{\rho_1}-\sqrt{\rho_0}}{\sqrt{\rho_t}}\in L^2(\rd\rho_t).
$$
Moreover this optimal curve $\partial_t\rho_t=\rho_t r_t$ has constant metric speed $\|r_t\|^2_{L^2(\rd\rho_t)}= 4\int_\Omega|\sqrt{\rho_1}-\sqrt{\rho_0}|^2=\H^2(\rho_0,\rho_1)$, which should be expected for geodesics.\\

More importantly, the first Lagrangian formulation in \eqref{eq:def_HFR} suggests to view the metric space $(\M^+,\H)$ as a Riemannian manifold, endowing the tangent plane
$$
T_\rho\M^+_\H=\Big\{\partial_t\rho_t=\rho_t r_t\qquad \mbox{evaluated at }t=0\Big\}
$$
with the metrics
$$
\|\partial_t\rho\|^2_{T_\rho\M^+_\H}=\|r\|^2_{L^2(\rd\rho)}
\quad \mbox{with the identification }\partial_t\rho = \rho r.
$$
Metric gradients $\grad_\H$ can then be computed by the chain rule as follows:
If $\partial_t\rho_t=\rho_t r_t$ is a { smooth} curve passing through $\rho_t(0)=\rho$ with arbitrary initial velocity $\zeta=\partial_t\rho=\rho r$ then for functionals $\mathcal{F}(\rho)=\int_\Omega F(\rho(x),x)\rd x$ we can compute
\begin{align*}
\left<\grad \mathcal F(\rho),\zeta\right>_{T_\rho\M^+_\H} & =\left.\frac{d}{dt}\mathcal F (\rho_t)\right|_{t=0} = \left.\frac{d}{dt}\left(\int_\Omega F(\rho_t(x),x)\rd x\right)\right|_{t=0}\\
& =\int_{\Omega}F'(\rho)\rho r=\left< F'(\rho),r \right>_{L^2(\rd\rho)},
\end{align*}
where $F'(\rho)=\frac{\delta F}{\delta\rho}$ as before.
This shows that
\begin{equation}
\grad_\H\mathcal F(\rho)=\rho F'(\rho)
\label{eq:formula_grad_HFR}
\end{equation}
with identification through the $L^2(\rd\rho)$ action in the tangent plane, and as a consequence gradients flows with respect to the Hellinger-Fisher-Rao metrics read
\begin{equation}
\partial_t\rho=-\grad_\H\mathcal F(\rho)
\qquad \leftrightarrow\qquad \partial_t\rho=-\rho F'(\rho).
\label{eq:formula_grad_flow_HFR}
\end{equation}
\subsection{The Kantorovich-Fisher-Rao distance $\d$}
\label{section:FRHKW}
As introduced in \cite{Peyre_inter_15}, we have
\begin{defi}
 The Fisher-Rao-Hellinger-Kantorovich-Wasserstein distance between measures $\rho_0,\rho_1\in \M^+(\Omega)$ is
 \begin{equation}
  \d^2(\rho_0,\rho_1)=\inf\limits_{(\rho,\v,r)\in \mathcal A_\d[\rho_0,\rho_1]}\int_0^1\int_\Omega (|\v_t(x)|^2+|r_t(x)|^2)\rd \rho_t(x)\,\rd t
  \label{eq:def_d}
 \end{equation}
 The admissible set $\mathcal A_\d[\rho_0,\rho_1]$ is the set of curves $ [0,1]\ni t\mapsto(\rho_t,\v_t,r_t)\in \mathcal M^+(\Omega)\times L^2(\Omega,\rd\rho_t)^d\times L^2(\Omega,\rd\rho_t)$ such that $ t\mapsto\rho_t$ is narrowly continuous with endpoints $ \rho_0,\rho_1$ and solves the continuity equation with source 
 $$
 \partial_t\rho_t+\dive(\rho_t\v_t)=\rho_t r_t
 $$
 in the sense of distributions $\mathcal D'((0,1)\times\Omega)$.
  \end{defi}
   { As in Remark~\ref{rmk:functional_space_L2(drho)}
the velocity fields and reaction term implicitly belong to the energy space $ L^2(0,1;L^2(\rd\rho_t))$, so that both products $\rho\v,\rho r$ are well-defined as distributions $\mathcal D'((0,1)\times\Omega)$.}
Comparing \eqref{eq:def_d} with \eqref{eq:benamou_brenier_formula} and \eqref{eq:def_HFR}, this dynamic formulation {\it \`a la Benamou-Brenier} \cite{BB_00} shows that the $\KFR$ distance can be viewed as an inf-convolution of the Monge-Kantorovich and Fisher-Rao distances $\W,\H$.
By the results of \cite{Peyre_inter_15,Peyre_unbalanced_15,LMS_big_15} the infimum in the definition is always a minimum, and the corresponding minimizing curves $t\mapsto\rho_t$ are of course called geodesics.
As shown in \cite{KMV_15,Peyre_inter_15,LMS_big_15} geodesics need not be unique, see also the brief discussion in section~\ref{section:minimizing_scheme}.
Interestingly, there are other possible formulations of the distance in terms of static unbalanced optimal transportation, primal-dual characterizations with relaxed marginals, lifting to probability measures on a cone over $\Omega$, and duality with subsolutions of Hamilton-Jacobi equations.
See also \cite{LMS_big_15,LMS_small_15} as well as \cite{piccoli_rossi_generalized_14} for a related version with mass penalization.

As an immediate consequence of the definition~\ref{eq:def_d} we have a first interplay between the distances $\d,\W,\H$:
\begin{prop}
\label{prop:comparison_d_W_H}
Let $\rho_0,\rho_1\in \M^+_2$ such that $|\rho_0|=|\rho_1|$. Then
$$
\d^2(\rho_0,\rho_1)\leq \W^2(\rho_0,\rho_1).
$$
Similarly for all $\mu_0,\mu_1\in \M^+$ (with possibly different masses) there holds
$$
\d^2(\mu_0,\mu_1)\leq \H^2(\mu_0,\mu_1).
$$
\end{prop}
\begin{proof}
If $|\rho_0|=|\rho_1|$ then the optimal Monge-Kantorovich geodesics $\partial_t\rho_t+\dive(\rho_t\v_t)=0$ from $\rho_0$ to $\rho_1$ gives an admissible path in \eqref{eq:def_d} with $r\equiv 0$ and cost exactly $\W^2(\rho_0,\rho_1)$. Likewise for arbitrary measures $\mu_0,\mu_1$ one can follow the Fisher-Rao geodesics $\partial_r\rho_t=\rho_t r_t$, which gives an admissible path with $\v\equiv 0$ and cost $\H^2(\mu_0,\mu_1)$.
\end{proof}
\begin{prop}
\label{prop:coupling=uncoupling_d}
The definition \eqref{eq:def_d} of the $\KFR$ distance can be restricted to the subclass of admissible paths $(\v_t,r_t)$ such that $\v_t=\nabla r_t$. 
\end{prop}
\begin{proof}
By \cite[thm. 2.1]{Peyre_inter_15} there exists a minimizing curve $(\rho_t,\v_t,r_t)$ in \eqref{eq:def_d}, which by definition is a $\KFR$-geodesic between $\rho_0,\rho_1$ (we also refer to \cite[thm. 6]{KMV_15} and \cite{LMS_small_15} for the existence of geodesics).
Here we stay at the formal level and assume that $\rho,\v,r$ are smooth with $\rho>0$ everywhere.

Consider first an arbitrary smooth vector-field $\mathbf w$ such that $\dive \mathbf{w}_t=0$ for all $t\in [0,1]$, and let $\v^{\eps}:=\v+\eps\frac{\mathbf{w}}{\rho}$.
Then $\dive(\rho\mathbf{v}^\eps)=\dive(\rho\v)+0$ and the triplet $(\rho_t,\mathbf v^\eps_t,r_t)$ is an admissible competitor in \eqref{eq:def_d}. Writing the optimality condition we compute
\begin{align*}
0 & =\left.\frac{d}{d\eps}\left(\frac{1}{2}\int_0^1\int_\Omega (|\mathbf{v}^\eps_t(x)|^2+ |r_t(x)|^2)\rd \rho_t(x)\,\rd t\right)\right|_{\eps=0} \\
& = \int_0^1\int_\Omega \v_t(x) \cdot \frac{\mathbf{w}_t(x)}{\rho_t(x)}\rd \rho_t(x)\,\rd t
=
\int_0^1\int_\Omega\v_t(x)\cdot \mathbf{w}_t(x)\,\rd x\rd t.
\end{align*}
This $L^2$ orthogonality with all divergence-free vector fields classically implies that $\v_t$ is potential for all times, i-e $\v_t=\nabla u_t$ for some $u_t$.

Fix now any smooth $\phi\in\mathcal{C}^\infty_c((0,1)\times \Omega)$, and define $\tilde\v^\eps_t:=\v_t+\eps\nabla\phi_t=\nabla(u_t+\eps\phi_t)$. Defining $s_t$ by $\rho_t s_t=\dive(\rho_t\nabla\phi_t)$ and $\tilde r_t^\eps:=r_t+\eps s_t$ it is easy to check that $(\rho_t,\tilde\v_t^\eps,\tilde r_t^\eps)$ solves the continuity equation, and this triplet is again an admissible competitor in \eqref{eq:def_d}. Writing the optimality condition we get now
\begin{multline*}
0  =\left.\frac{d}{d\eps}\left(\frac{1}{2}\int_0^1\int_\Omega (|\tilde \v^\eps_t(x)|^2 + |\tilde r^\eps_t(x)|^2)\rd \rho_t(x)\,\rd t\right)\right|_{\eps=0} \\
 = \int_0^1\int_\Omega\Big( \nabla u_t(x)\cdot\nabla\phi_t +  r_t(x)s_t(x)\Big)\rd \rho_t(x)\,\rd t\\
 =\int_0^1\int_\Omega \nabla\Big(u_t-r_t\Big)(x)\cdot\nabla\phi_t\rd \rho_t(x)\,\rd t,
\end{multline*}
 where we used the identity $ r_t s_t\rho_t = r_t\dive(\rho_t\nabla\phi_t)$ to integrate by parts in the last equality.
 As $\phi$ was arbitrary this implies $\dive(\rho_t\nabla u_t)=\dive(\rho_t\nabla r_t)$ and $\|\v_t\|^2_{L^2(\rd\rho_t)}=\|\nabla u_t\|^2_{L^2(\rd\rho_t)}=\|\nabla r_t\|^2_{L^2(\rd\rho_t)}$.
 In particular the triplet $(\rho_t,\nabla r_t,r_t)$ is admissible and has the same cost as the optimal $(\rho_t,\v_t,r_t)$, which concludes the proof.
\end{proof}
As a consequence we have the alternative definition of the $\KFR$ distance as introduced in \cite{KMV_15}, which couples the reaction and velocity:
\begin{theo}
\label{theo:KFR_coupled_uncoupled}
 For all $\rho_0,\rho_1\in \M^+(\Omega)$ there holds
 \begin{equation}
  \d^2(\rho_0,\rho_1)=\inf\limits_{(\rho,u)\in \tilde{\mathcal A}_\d[\rho_0,\rho_1]}\int_0^1\int_\Omega (|\nabla u_t(x)|^2+|u_t(x)|^2)\rd \rho_t(x)\,\rd t,
  \label{eq:def_d_coupled}
 \end{equation} 
 where $\tilde A_\d[\rho_0,\rho_1]$ is the set of weakly continuous curves $t\mapsto \rho_t\in\mathcal C_w ([0,1];\M^+)$ with endpoints $ \rho_0,\rho_1$ such that
 $$
 \partial_t\rho_t+\dive(\rho_t\nabla u_t)=\rho_t u_t
 $$
 in the sense of distributions $\mathcal D'((0,1)\times\Omega)$.
\end{theo}
{ The potentials $u$ belong now implicitly to the energy space $L^2(0,1;H^1(\rd\rho_t))$ with obviously $\|u_t\|^2_{H^1(\rd\rho)}:=\int_\Omega( |\nabla u_t|^2+ |u_t|^2)\rd\rho_t$, and both products $\rho_t\nabla u_t,\rho_t u_t$ define distributions as before.
Note that Theorem~\ref{theo:KFR_coupled_uncoupled} shows that the $\KFR$ distance constructed in \cite{Peyre_inter_15}, based on the uncoupled $(\v,r)$ formulation, is indeed the same as that in \cite{KMV_15}, modeled on the $(\nabla u,u)$ potential framework.

In order to define now the Riemannian structure on $(\M^+,\d)$ inherited from the Lagrangian minimization, 
we endow the tangent plane}
{ 
$$
T_\rho\mathcal{M}^+_\d=\Big\{\partial_t\rho = - \dive(\rho v)+ \rho r\quad \mbox{evaluated at $t=0$}\Big\}
$$
with the Riemannian metrics
$$
\|\partial_t\rho\|^2_{T_\rho\M^+_{\d}}: =\inf\left\{\|\v\|^2_{L^2(\rd\rho)}+\|r \|^2_{L^2(\rd\rho)} :\quad \partial_t\rho=-\dive(\rho\v) + \rho r \right\}.
$$
Then Theorem~\ref{theo:KFR_coupled_uncoupled} also allows to construct the one-to-one correspondence between tangent vectors $\partial_t\rho$ and potentials $u$, such that
$$
\|\partial_t\rho\|^2_{T_\rho\M^+_{\d}} =\|u\|^2_{H^1(\rd\rho)}\quad \mbox{with the identification }\partial_t\rho=-\dive(\rho\nabla u)+ \rho u.
$$
}
{ 
With this one-to-one correspondence at hand,} metric gradients $\grad_\d\mathcal F$ can be computed by the chain rule as earlier:
If $\partial_t\rho_t+\dive(\rho_t\nabla u_t)=\rho_t u_t$ is a { smooth} curve passing through $\rho_t(0)=\rho$ with arbitrary initial velocity $\zeta=\partial_t\rho_t(0)=-\dive(\rho \nabla u)+\rho u$ then for functionals $\mathcal{F}(\rho)=\int_\Omega F(\rho(x),x)\rd x$ we have
\begin{multline*}
\left<\grad_\d \mathcal F(\rho),\zeta\right>_{T_\rho\M^+_\d}  =\left.\left.\frac{d}{dt}\mathcal F (\rho_t)\right|_{t=0} = \frac{d}{dt}\left(\int_\Omega F(\rho_t(x),x)\rd x\right)\right|_{t=0}\\
 =\int_{\Omega}F'(\rho)\times \left\{-\dive(\rho\nabla u)+\rho u\right\}\\
 = \int_\Omega\left\{\nabla F'(\rho)\cdot \nabla u+F'(\rho) u\right\}\rd\rho
 =\left< F'(\rho),u \right>_{H^1(\rd\rho)},
\end{multline*}
where $F'(\rho)=\frac{\delta F}{\delta\rho}$ as before.
This shows that
\begin{equation*}
 \grad_\d\mathcal F(\rho)=- \dive\left(\rho\nabla F'(\rho)\right) + \rho F'(\rho)
 \label{eq:formula_gradient_d}
\end{equation*}
through the canonical $H^1(\rd\rho)$ action in the tangent plane.
In particular $\d$ gradient flows read
\begin{equation}
\partial_t\rho =-\grad_\d \mathcal F(\rho)
\qquad \leftrightarrow \qquad
\partial_t\rho=\dive(\rho \nabla F'(\rho))-\rho F'(\rho),
\label{eq:formula_grad_flow_FRHKW}
\end{equation}
which should be compared with \eqref{eq:formula_grad_flow_W} and \eqref{eq:formula_grad_flow_HFR}.
%
\section{Infinitesimal uncoupling of the inf-convolution}
\label{section:uncoupling_inf-sup_convolution}
Let us first summarize the previous informal discussion on each of the three metrics: the quadratic Monge-Kantorovich distance is modeled on the homogeneous $\dot H^1(\rd\rho)$ space, the Fisher-Rao distance is based on $L^2(\rd\rho)$, and the $\KFR$ metrics is constructed on the full $H^1 (\rd \rho)$ structure.
Each of these Riemannian structures are defined via identification of tangent vectors as
\begin{equation*}
\begin{array}{llr}
 \W:\quad & \|\partial_t\rho\|^2_{T_\rho\M^+_\W}=\| \nabla p\|^2_{L^2(\rd\rho)}=\int_\Omega |\nabla p|^2\rd\rho, & \partial_t\rho +\dive(\rho\nabla p)=0,\\
 \H:\quad & \|\partial_t\rho\|^2_{T_\rho\M^+_\H}=\| r\|^2_{ L^2(\rd\rho)}=\int_\Omega |r|^2\rd\rho, & \partial_t\rho =\rho r,\\
\d:\quad & \|\partial_t\rho\|^2_{T_\rho\M^+_{\d}}=\| u\|^2_{ H^1(\rd\rho)}=\int_\Omega (|\nabla u|^2+u^2)\rd\rho,\quad & \partial_t\rho +\dive(\rho\nabla u)=\rho u.
\end{array}
\end{equation*}
{ 
Given a tangent vector $\zeta^u_\KFR=-\dive(\rho\nabla u)+\rho u\in T_\rho\mathcal M^+_{\KFR}$ we can naturally define a Monge-Kantorovich tangent vector $\zeta^u_\MK:=-\dive(\rho\nabla u)\in T_\rho \mathcal M^+_\MK$, and a Fisher-Rao tangent vector $\zeta^u_\FR:=\rho u\in T_\FR \mathcal M^+_\FR$.
Observing that by construction
\begin{equation}
\label{eq:orthogonality_norms}
\|\zeta^u_\KFR\|^2_{T_\rho\mathcal M^+_\KFR}= \|\zeta^u_\MK\|^2_{T_\rho\mathcal M^+_\MK}+\|\zeta^u_\FR\|^2_{T_\rho\mathcal M^+_\FR},
\end{equation}
this suggests to view the tangent plane as the orthogonal sum
\begin{equation}
T_{\rho}\M^+_{\d} = T_{\rho}\M^+_\W \oplus^\perp T_{\rho}\M^+_\H,\qquad
\zeta^u_\KFR=\zeta^u_\MK+\zeta^u_\FR.
\label{eq:orthogonal_decomposition}
\end{equation}
More precisely, let us define an equivalence relation $\thicksim$ on $T_{\rho}\M^+_\W \oplus T_{\rho}\M^+_\H$ by $(\v,r) \thicksim (\tilde{\v} , \tilde r) $ if 
$  -\dive \left(\rho \v\right)+\rho  r= -\dive \left(\rho \tilde{\v} \right)+ \rho  \tilde r $.
Each $(\v, r)$ lies in an equivalence class $\left[ (\nabla u,u)\right] = [u]$ on which we  define the norm 
$$
\| [u] \|^2_{\thicksim}  = \| \nabla u \|^2_{ L^2(\rd\rho)}  +\| u\|^2_{ L^2(\rd\rho)} = \|\zeta^u_\MK\|^2_{T_\rho\mathcal M^+_\MK}+\|\zeta^u_\FR\|^2_{T_\rho\mathcal M^+_\FR}. 
$$
Then the orthogonality in \eqref{eq:orthogonality_norms}  should be understood as
$$
\Big( T_{\rho}\M^+_{\d},\| \cdot \|^2_{T_\rho\M^+_{\d}} \Big) = \Big( \left(  T_{\rho}\M^+_\W \oplus  T_{\rho}\M^+_\H \right) / \thicksim, \| \cdot \|^2_{\thicksim} \Big).
$$
Thus infinitesimally $\d^2\approx \W^2+\H^2$, and this will motivate later on replacing the approximation ``by hypotenuses'' by an approximation ``by legs'' in the JKO scheme - see section~\ref{section:minimizing_scheme} and in particular \eqref{eq:step_wasserstein}\eqref{eq:step_hellinger_fisher_rao}.
The orthogonality between the transport/$\MK$ and reaction/$\FR$ processes }also yields a natural strategy to send a measure $\rho_0$ to another $\rho_1$: one can send first $\rho_0$ to the renormalized $\tilde{\rho}_0:=\frac{|\rho_0|}{|\rho_1|}\rho_1$ by pure Monge-Kantorovich transport (which is possible since $|\tilde\rho_0|=|\rho_0|$), and then send $\tilde\rho_0$ to $\rho_1$ by pure Fisher-Rao reaction.
This amounts to following separately and successively the two orthogonal directions in the decomposition \eqref{eq:orthogonal_decomposition}.

An immediate consequence of this observation is
\begin{prop}
\label{prop:uncoupling_H2<2W2+2H2}
For arbitrary measures $\rho_0,\rho_1\in \M^+$ let $\tilde{\rho}_0:=\frac{|\rho_0|}{|\rho_1|}\rho_1$. Then
\begin{equation}
 \label{eq:estimate_d2<W2+d2}
 \d^2(\rho_0,\rho_1)\leq 2\big(\W^2(\rho_0,\tilde{\rho}_0)+\H^2(\tilde{\rho}_0,\rho_1)\big).
\end{equation}
\end{prop}
\begin{proof}
It suffices to follow first a pure Monge-Kantorovich geodesics ($r\equiv 0$) from $\rho_0$ to $\tilde{\rho}_0$ scaled in time $t\in [0,1/2]$, and then a pure Fisher-Rao geodesic ($\v \equiv 0$) from $\tilde{\rho}_0$ to $\rho_1$ scaled in time $t\in [1/2,1]$.
Because of the rescaling in time each of these half-paths have an extra factor $2$, amounting to a total cost of $2\W^2(\rho_0,\tilde \rho_0)+2\H^2(\tilde\rho_0,\rho_1)$ for this admissible path.
The result then follows from the definition \eqref{eq:def_d} of $\d^2$ as an infimum over all paths.
\end{proof}
Note that estimate \eqref{eq:estimate_d2<W2+d2} holds for any arbitrary measure $\rho_0,\rho_1\in \M^+$, but has a multiplicative factor $2$ which { in view of \eqref{eq:orthogonality_norms}\eqref{eq:orthogonal_decomposition} is certainly not optimal at short range $\KFR(\rho_0,\rho_1)\ll 1$}.
Consider now two very close measures $\d(\rho_0,\rho_1)\ll 1$. 
Then the above transformation from $\rho_0$ to $\rho_1$ can essentially be considered as occurring infinitesimally in the tangent plane $T_{\rho}\M^+_{\d} = T_{\rho}\M^+_\W \oplus^\perp T_{\rho}\M^+_\H$. Roughly speaking, this means that the two transport and reaction processes from $\rho_0$ to $\tilde{\rho}_0$ and from $\tilde{\rho_0}$ to $\rho_1$ in the previous proof can be considered as occurring \emph{simultaneously and independently} at the infinitesimal level.
Thus the factor $2$ in \eqref{eq:estimate_d2<W2+d2} is unnecessary, and one should expect in fact
\begin{equation}
\d^2(\rho_0,\rho_1)\approx \W^2(\rho_0,\tilde \rho_0)+\H^2(\tilde\rho_0,\rho_1)
\label{eq:claim_uncoupling_inf_sup_convolution}
\end{equation}
for { nearby} measures $\d(\rho_0,\rho_1)\ll 1$.
This can be made rigorous at least for one-point particles
$$
\rho_0=k_0\delta_{x_0},\qquad \rho_1=k_1\delta_{x_1}
$$
at close distance, i-e $|x_1-x_0|\ll 1$ and $k_1\approx k_0$.
{ In this setting it was shown in \cite[Section 3.3]{KMV_15} and proved rigorously \cite[thm. 4.1]{Peyre_inter_15} and \cite[thm. 3.1]{LMS_small_15} that the geodesics $\rho_t$ from $\rho_0$ to $\rho_1$ is a moving one-point mass of the form $\rho_t=k_t\delta_{x_t}$ for some suitable curve $t\mapsto(x_t,k_t)\in \Omega\times\R^+$.} 
\begin{rmk}
\label{rmk:threshold_pi}
The one-point ansatz $\rho_t=k_t\delta_{x_t}$ is in fact correct not only for short distances $|x_1-x_0|\ll1$, but also as long as $|x_1-x_0|<\pi$. Past this threshold $|x_1-x_0|=\pi$ it is more efficient to virtually displace mass from $x_0$ to $x_1$ by pure reaction, i-e by killing mass at $x_0$ while simultaneously creating some at $x_1$.
\end{rmk}
In the continuity equation $\partial_t\rho_t+\dive(\rho_t\v_t)=\rho_t r_t$ the advection moves particles around according to $\frac{d}{dt}x_t=\v_t$ and the reaction reads $\frac{d}{dt}k_t=k_tr_t$, each with infinitesimal cost $k_t|\v_t|^2$ and $k_t|r_t|^2$. 
The optimal $(\v_t,r_t)$ for the one-point ansatz $\rho_t=k_t\delta_{x_t}$ can be computed explicitly by looking at the coupled formulation \eqref{eq:def_d_coupled} with $\v_t=\nabla u_t,r_t=u_t$, and optimizing the cost with respect to admissible potentials $u_t$.
Omitting the details (see again \cite{Peyre_unbalanced_15,Peyre_inter_15,KMV_15,LMS_big_15,LMS_small_15}), the optimal cost can be computed explicitly as
\begin{equation}
\d^2(\rho_0,\rho_1)=4\left(k_0+k_1-2\sqrt{k_0k_1}\cos\left(\frac{|x_1-x_2|}{2}\right)\right)
\qquad \mbox{for }|x_1-x_0|<\pi.
\label{eq:formula_d_one_point_particles}
\end{equation}
\begin{rmk}
{ 
It was shown in \cite{Peyre_unbalanced_15,LMS_big_15,LMS_small_15} that the $\KFR$ distance can be recovered by means of a suitable Riemannian submersion 
$(\mathcal P_2(C_\Omega),\W)\to (\M^+(\Omega),\d)$.
Here $\mathrm C_\Omega=\{[x,r]\in \Omega\times\R^+\}/\thicksim$ is a cone overlying $\Omega$ obtained by identification of all the tips $[x,0]$ into a single point $\diamond\in C_\Omega$, and is suitably endowed with the cone distance $d_C^2([x_0,r_0],[x_1,r_1])=r_0^2+r_1^2-2r_0r_1\cos(|x_1-x_0|/2 \wedge \pi)$. 
In formula \eqref{eq:formula_d_one_point_particles} one sees in fact, up to the normalizing factor $4$, the natural Monge-Kantorovich distance $\d^2\left(\delta_{[x_0,{k_0}]},\delta_{[x_1,{k_1}]}\right)= \W^2\left(\delta_{[x_0,\sqrt{k_0}]},\delta_{[x_1,\sqrt{k_1}]}\right)=d_C^2([x_0,\sqrt{k_0}],[x_1,\sqrt{k_1}])$ between unit Dirac masses in the overlying space $\mathcal P_2(C_\Omega)$.
We refrain from discussing further the Riemannian submersion and the corresponding static formulations of $\KFR$, and refer again to \cite{Peyre_unbalanced_15,LMS_big_15,LMS_small_15,TGV} for rigorous statements.
}
\end{rmk}

In this setting and with the previous notation $\tilde{\rho}_0=\frac{|\rho_0|}{|\rho_1|}\rho_1=k_0\delta_{x_1}$ we have here
$$
\W^2(\rho_0,\tilde\rho_0)=\W^2(k_0\delta_{x_0},k_0\delta_{x_1})=k_0|x_1-x_0|^2,
$$
and by \eqref{eq:def_HFR}
$$
\H^2(\tilde{\rho}_0,\rho_1)=4\int_\Omega\left|\sqrt{\frac{\rd \rho_1}{\rd\delta_{x_1}}}-\sqrt{\frac{\rd \tilde\rho_0}{\rd\delta_{x_1}}}\right|^2\rd\delta_{x_1}=4\left|\sqrt{k_1}-\sqrt{k_0}\right|^2.
$$
Taylor-expanding \eqref{eq:formula_d_one_point_particles} at order two in $|x_1-x_0|,|\sqrt{k_1}-\sqrt{k_0}|\ll 1$ gives
\begin{multline}
\d^2(\rho_0,\rho_1)  =k_0 |x_1-x_0|^2  + 4 |\sqrt{k_1}-\sqrt{k_0}|^2  +\mathcal O\Big(|x_1-x_0|^2|\sqrt{k_1}-\sqrt{k_0}|\Big)
\label{eq:explicit_computation_uncoupling_inf-sup_one-point}
\\
= \W^2(\rho_0,\tilde{\rho}_0)  + \H^2(\tilde\rho_0,\rho_1)  +\mbox{lower order},
\end{multline}
which shows that our claim \eqref{eq:claim_uncoupling_inf_sup_convolution} holds true at least for one-point particles and at order one in the squared distances.
\begin{rmk}
Due to $4|\sqrt{k_1}-\sqrt{k_0}|^2=\FR^2(\tilde\rho_0,\rho_1)\ll 1$ we have $k_1=k_0+\mathcal O(|\sqrt{k_1}-\sqrt{k_0}|)$.
The previous expression can therefore be rewritten as 
$$
\d^2(\rho_0,\rho_1)  =\frac{k_0+k_1}2 |x_1-x_0|^2  + 4 |\sqrt{k_1}-\sqrt{k_0}|^2  +\mbox{lower order}
$$
and the apparent loss of symmetry in $k_0,k_1$ in \eqref{eq:explicit_computation_uncoupling_inf-sup_one-point} is thus purely artificial.
\end{rmk}

\begin{rmk}
An interesting question would be to determine how much information on the transport/reaction coupling is encoded in the remainder,
{ and this is also related to the curvature of the $\KFR$ space.}
\end{rmk}
Justifying and/or quantifying the above discussion and \eqref{eq:claim_uncoupling_inf_sup_convolution} for general measures with $\d(\rho_0,\rho_1)\ll 1$ is an interesting question left for future work. 
One can think that the superposition principle should apply: viewing any measure as a continuum of one-point Lagrangian particles and taking for granted that the infinitesimal uncoupling holds for single particles, it seems natural that the result should also hold for all measures.

\section{Minimizing scheme}
\label{section:minimizing_scheme}
We turn now our attention to gradient-flows
\begin{equation}
\partial_t\rho=-\grad_\d \mathcal F(\rho) 
\label{eq:grad_flow_d}
\end{equation}
of functionals
$$
\mathcal{F}(\rho)=
\left\{
\begin{array}{ll}
\int_{\Omega}\big\{U(\rho)+\Psi(x)\rho+\frac 12 \rho K\star \rho\big\}\rd x \quad & \mbox{if }\rd\rho\ll\rd x\\
\infty & \mbox{otherwise}
\end{array}
\right.
$$
with respect to the $\KFR$ distance.
Without further mention we implicitly restrict to absolutely continuous measures (with respect to Lebesgue), and still denote their Radon-Nikodym derivatives $\rho=\frac {\rd\rho}{\rd x}$ with a slight abuse of notations.
According to \eqref{eq:formula_grad_flow_FRHKW} this corresponds to PDEs of the form
\begin{equation}
\partial_t \rho =\dive(\rho \nabla(U'(\rho)+\Psi+K\star \rho))-\rho(U'(\rho)+\Psi+K\star \rho),
\label{eq:PDE}
\end{equation}
appearing for example in the tumor growth model studied in \cite{perthame_tumor_14}.

The natural minimizing movement for \eqref{eq:grad_flow_d} should be
\begin{equation}
\rho^{n+1}\in \underset{\rho\in\M^+}{\operatorname{Argmin}}\left\{\frac{1}{2\tau}\d^2 (\rho,\rho^n)+\mathcal F(\rho)\right\}
\label{eq:minimizing_scheme_d}
\end{equation}
for some small time step $\tau>0$.
In order to obtain an Euler-Lagrange equation, a classical and natural strategy would be to consider perturbations $\eps\mapsto\rho_\eps$ of the minimizer $\rho_\eps(0)=\rho^{n+1}$ starting with velocity $\partial_\eps\rho_\eps(0)=-\dive(\rho^{n+1}\nabla\phi)+\rho^{n+1}\phi$ for any arbitrary smooth $\phi$, corresponding to choosing all possible directions of perturbation in the tangent plane $T_{\rho^{n+1}}\M^+_{\d}$.
Writing down the optimality criterion $\left.\frac{d}{d\eps}\left(\frac{1}{2\tau}\d^2(\rho_\eps,\rho^n)+\mathcal F(\rho_\eps)\right)\right|_{\eps=0}=0$ should then give the sought Euler-Lagrange equation. 
In order to exploit this, one should in particular know how to differentiate the squared distance $\rho\mapsto \d^2(\rho,\mu)$ with respect to such perturbations $\rho_\eps$ of the minimizer.
At this stage the theory does not provide yet the necessary tools, even though what the formula should be is quite clear: For any reasonable smooth Riemannian manifold and curve $x(t)$ with $x(0)=x$ we have
$$
\left.\frac{d}{dt}\left(\frac{1}{2} d^2(x(t),y)\right)\right|_{t=0}=\left<x'(0),\zeta\right>_{T_x\mathcal M},
$$
where $\zeta$ is the terminal velocity $y'(1)\in T_x\mathcal M$ of the geodesics from $y$ to $x$. 
Here the $\d$-geodesics $(\mu_s)_{s\in [0,1]}$ from $\mu$ to $\rho$ should solve $\partial_s\mu_s+\dive(\mu_s\nabla u_s)=\mu_s u_s$ and the terminal velocity $\zeta=\partial_s\mu(1)\in T_\rho \M^+_{\KFR}$ should be identified with some potential $u=u_s(1)\in H^1(\rd\rho)$ through $\zeta=-\dive(\rho\nabla u)+\rho u$, see section~\ref{section:FRHKW}.
We should therefore expect
$$
\left.\frac{d}{d\eps}\left(\frac{1}{2}\d^2(\rho_\eps,\mu)\right)\right|_{\eps=0}=\left<\partial_\eps\rho(0),\zeta\right>_{T_\rho\M^+_{\d}}=\int_\Omega (\nabla \phi\cdot \nabla u+ \phi u)\rd\rho.
$$
However, this can raise delicate technical issues at the cut-locus, where geodesics cease to be minimizing and prevent any differentiability of the squared distance.
%
Indeed, it was shown in \cite[section 5.2]{LMS_small_15}, \cite[thm. 4.1]{Peyre_inter_15}, and \cite[section 3.5]{KMV_15} that such cut-loci do exist for $\Omega=\R^d$, and even that the set of non-unique geodesics generically spans an infinite-dimensional convex set.
This is related to the threshold $|x_1-x_0|=\pi$ for one-point measures, see Remark~\ref{rmk:threshold_pi}.
In other words the squared distance may very well not be differentiable, even in the case of the simplest geometry $\Omega=\R^d$ of the underlying space.
This is in sharp contrast with classical mass conservative optimal transportation, where the cut-locus in $\mathcal P(X)$ is intimately related to the geometry of the underlying Riemannian manifold $X$ \cite{villani_big}.

In the context of minimizing movements one should expect two successive steps to be extremely close, typically $\d(\rho^{n+1},\rho^n)=\mathcal O(\sqrt{\tau})$ as $\tau\to 0$.
It seems reasonable to hope that geodesics then become unique at short distance, and one might therefore think that the previous cut-locus issue should not arise here for small $\tau>0$.
However, even assuming that we could somehow compute a unique minimizing geodesics $(\rho_s)_{s\in [0,1]}$ from $\rho^n$ to $\rho^{n+1}$ and safely evaluate the terminal velocity $\partial_s\rho(1)=-\dive(\rho^{n+1}\nabla u^{n+1})+\rho^{n+1} u^{n+1}$ at $s=1$ in order to differentiate the squared distance, it would remain to derive a (possibly approximated) relation between the Riemannian point of view and the more classical PDE framework, e.g. by proving an estimate like
$$
\int_{\Omega}(\nabla u^{n+1}\cdot\nabla\phi+u^{n+1}\phi)\rd \rho^{n+1}
\approx
\int_{\Omega}\frac{\rho^{n+1}-\rho^n}{\tau}\phi + \mbox{remainder}.
$$
In this last display we see the interplay between the forward tangent vector $u^{n+1}\in H^1(\rd\rho^{n+1})\leftrightsquigarrow T_{\rho^{n+1}}\M^+_{\d}$, encoding the Riemannian variation from $\rho^n$ to $\rho^{n+1}$, and the standard difference quotient $\frac{\rho^{n+1}-\rho^n}{\tau}\approx\partial_t\rho$.
One should then typically prove that the remainder is quadratic $\mathcal O\left(\d^2(\rho^{n+1},\rho^n)\right)$.
Within the framework of classical optimal transport this is usually done exploiting the explicit representation of the $\W$ metrics in terms of optimal transport maps (or transference plans, or Kantorovich potentials), which are in turn related to some static formulations of the problem. 
See later on section \ref{section:Wasserstein_substep} and in particular the Taylor expansion \eqref{eq:Taylor_expansion_W2} for details, and also remark~\ref{rmk:no_asymptotic_expansion_Hellinger}.
However, and even though static formulations of the $\KFR$ distance have been derived in \cite{LMS_big_15}, the current theory does not provide yet such an asymptotic expansion.

In order to circumvent these technical issues, let us recall from the discussion in section~\ref{section:uncoupling_inf-sup_convolution} that the inf-convolution formally uncouples at short distance.
This strongly suggests replacing $\d^2$ by the approximation $\W^2+\H^2\approx \d^2$, and as a consequence we naturally substitute the direct one-step minimizing scheme \eqref{eq:minimizing_scheme_d} by a sequence of two elementary substeps
$$
\rho^n\overset{ \W^2}{\longrightarrow }\rho^{n+\frac{1}{2}}\overset{\H^2}{\longrightarrow } \rho^{n+1}.
$$
Each of these substeps are pure Monge-Kantorovich/transport and Fisher-Rao/reaction variational steps, respectively and successively
\begin{equation}
\rho^{n+\frac{1}{2}}\in \underset{\rho\in\M_2^+,\,|\rho|=|\rho^n|}{\operatorname{Argmin}}\left\{\frac{1}{2\tau}\W^2(\rho,\rho^n)+\mathcal F(\rho)\right\}
\label{eq:step_wasserstein}
\end{equation}
\begin{equation}
\rho^{n+1}\in \underset{\rho\in\M^+}{\operatorname{Argmin}}\left\{\frac{1}{2\tau}\H^2(\rho,\rho^{n+\frac{1}{2}})+\mathcal F(\rho)\right\}.
\label{eq:step_hellinger_fisher_rao}
\end{equation}
Note that the first Monge-Kantorovich step is mass preserving by construction, while the second will account for mass variations.

The underlying idea is that the scheme follows alternatively the two privileged directions in $T_\rho\M^+_{\d}=T_\rho\M^+_\W\oplus T_\rho\M^+_\H$, corresponding to pure Monge-Kantorovich transport and pure Fisher-Rao reaction respectively. 
Another possible interpretation is that of an operator-splitting method: from \eqref{eq:formula_grad_flow_W}\eqref{eq:formula_grad_HFR}\eqref{eq:formula_grad_flow_FRHKW} we get
\begin{align*}
-\grad_\d\mathcal F(\rho) & =\dive(\rho\nabla(U'(\rho)+\Psi+K\star \rho)) -\rho(U'(\rho)+\Psi+K\star \rho)\\
& =-\grad_{\W}\mathcal F (\rho) -\grad_\H\mathcal F(\rho).
\end{align*}
Viewing the same functional $\mathcal F(\rho)$ through distinct ``differential lenses'' (i-e using respectively the $\W$ and $\H$ differential structures) gives the two transport and reaction terms in the PDE \eqref{eq:PDE}.
Thus it is very natural to split the PDE in two separate transport/reaction operators and treat separately each of them in their own and intrinsic differential framework.
This idea of hybrid variational structures has been successfully applied e.g. in \cite{KK_janossy_09,blanchet_PKS_hybrid_15,blanchet_parabolic_parabolic_PKS_13} for systems of equations where each component is viewed from separate differential perspectives, but not to the splitting of one single equation as it is the case here.
A related splitting scheme was employed in \cite{Agueh_bowles_15} to construct weak solutions of fractional Fokker-Planck equations $\partial_t\rho=\Delta^{2s}\rho+\dive(\rho\nabla\Psi)$, using a Monge-Kantorovich variational scheme in order to handle the transport term. However the discretization of the fractional Laplacian was treated in a non metric setting, the PDE cannot be viewed as the sum of gradient-flows of the same functional for two different ``orthogonal'' metrics, and the approach therein is thus more a technical tool than an intrinsic variational feature.

Another natural consequence of this formal point of view is the following:
From the { orthogonality \eqref{eq:orthogonality_norms} in $T_{\rho}\M^+_{\d}=T_{\rho}\M^+_\W\oplus T_{\rho}\M^+_\H$ we can compute}
\begin{align*}
\mathcal{D}(t) :=-\frac{d}{dt}\mathcal F(\rho(t))=-\|\grad _d\mathcal F\|^2_{T_{\rho}\M^+_{\d}}
 = -\|\grad _\W\mathcal F\|^2_{T_{\rho}\M^+_\W}-\|\grad _\H\mathcal F\|^2_{T_{\rho}\M^+_\H},
\end{align*}
which really means that the total dissipation for the coupled $\d$ metrics is just the sum of the two elementary $\W,\H$ dissipations.
One can of course check this formula by computing $\frac{d}{dt}\mathcal F(\rho_t) $ along solutions of the PDE.
This may be useful at the discrete level, since regularity is essentially related to dissipation.
For example $\lambda$-convexity ensures that the energy is dissipated at a minimum rate, which in turn can be viewed as a quantifiable regularization in the spirit of Br\'ezis-Pazy.
This will be illustrated in Proposition~\ref{prop:EDI}, where we show that one indeed recovers an Energy Dissipation Inequality with respect to $\d$ from the two elementary $\W,\H$ geodesic convexity and dissipation.\\
{ We first collect some general properties of our two-steps $\W/\H$ splitting scheme, which share common features with the intrinsic one-step scheme \eqref{eq:minimizing_scheme_d} {  and only exploit the metric structure regardless of any PDE considerations.}
\begin{lem}[Total-square distance estimate]
 Let $\rho^n,\rho^{n+\frac{1}{2}}$ be recursive solutions of \eqref{eq:step_wasserstein}\eqref{eq:step_hellinger_fisher_rao}.
 Then
 \begin{equation}
\frac{1}{\tau}\sum\limits_{n\geq 0} \d^2(\rho^{n+1},\rho^n)\leq 4\left(\mathcal{F}(\rho^0)-\inf\limits_{\M^+}\mathcal F\right).
 \label{eq:total_square_distance_d}
\end{equation}
\end{lem}
Note that this estimate is useful only if $\mathcal F(\rho^0)<\infty$ and $\mathcal F$ is bounded from below.
The former condition is a natural restriction to finite energy initial data, and the latter is a reasonable assumption which holds true e.g. if $U(\rho)=\rho^m$ for some $m>1$ and the external potential $\Psi(x)\geq 0$ outside of a finite measure set.
\begin{proof}
Testing $\rho=\rho^n$ in \eqref{eq:step_wasserstein} and $\rho=\rho^{n+\frac 12}$ in \eqref{eq:step_hellinger_fisher_rao} we get
\begin{align*}
\frac{1}{2\tau}\W^2(\rho^{n+\frac{1}{2}},\rho^n)+\mathcal F(\rho^{n+\frac 12}) & \leq \mathcal F(\rho^{n}),\\
\frac{1}{2\tau}\H^2(\rho^{n+1},\rho^{n+\frac{1}{2}})+\mathcal F(\rho^{n+1}) & \leq \mathcal F(\rho^{n+\frac{1}{2}}).
\end{align*}
Summing over $n\geq 0$ and noticing that the energy contributions are telescopic, we get the mixed total-square distance estimate
\begin{equation}
\frac{1}{\tau}\sum\limits_{n\geq 0} \Big\{  \H^2(\rho^{n+1},\rho^{n+\frac{1}{2}}) +  \W^2(\rho^{n+\frac{1}{2}},\rho^n)  \Big\}\leq 2\left(\mathcal{F}(\rho^0)-\inf\limits_{\M^+}\mathcal F\right).
\label{eq:total_square_distance_hybrid}
\end{equation}
By triangular inequality and Proposition~\ref{prop:comparison_d_W_H} it is easy to check that
\begin{equation}
\d^2(\rho^{n+1},\rho^n)\leq 2\left(\H^2(\rho^{n+1},\rho^{n+\frac{1}{2}}) + \W^2(\rho^{n+\frac{1}{2}},\rho^n)\right),
\label{eq:discrete_d2<2W2+2H2}
\end{equation}
and our statement follows.
\end{proof}
\begin{rmk}
\label{rmk:total_square_dist_2_functionals}
 It is worth stressing that, when trying to handle two different functionals $\partial_t\rho =\dive(\rho \nabla F_1'(\rho))-\rho F_2'(\rho)$ in the diffusion and reaction, the distance estimate for the two successive $\MK,\FR$ steps would not result in a telescopic sum $\mathcal F(\rho^{n+1})-\mathcal F(\rho^{n+\frac 12})+\mathcal F(\rho^{n+\frac 12})-\mathcal F(\rho^{n})$ as above, but rather in $\mathcal F_1(\rho^{n+1})-\mathcal F_1(\rho^{n+\frac 12})+\mathcal F_2(\rho^{n+\frac 12})-\mathcal F_2(\rho^{n})$.
This can in fact be controlled with suitable compatibility conditions on $\mathcal F_1,\mathcal F_2$ and estimating the crossed dissipations as in \cite{theseMax,GLM_16}, but we decided to focus here on $\mathcal F_1=\mathcal F=\mathcal F_2$ in order to illustrate the general idea in a simpler variational setting.
\end{rmk}

As already discussed the factor $2$ in \eqref{eq:discrete_d2<2W2+2H2} is not optimal, and from the infinitesimal decoupling we should expect $\d^2(\rho^{n+1},\rho^n)\approx  \H^2(\rho^{n+1},\rho^{n+\frac{1}{2}}) + \W^2(\rho^{n+\frac{1}{2}},\rho^n)$.
Thus our estimate \eqref{eq:total_square_distance_d} should have a factor $2$ instead of $4$ in the right-hand side, which is exactly the classical total square distance estimate that one would get applying the direct one-step minimizing scheme \eqref{eq:minimizing_scheme_d} with respect to the full $\d$ metric.
\\

Assuming that we can solve recursively \eqref{eq:step_wasserstein}-\eqref{eq:step_hellinger_fisher_rao} for some given initial datum
$$
\rho_0\in \M^+,\qquad \mathcal{F}(\rho^0)<\infty,
$$
we construct two piecewise-constant interpolating curves
\begin{equation*}
\label{eq:def_piecewise_interpolations} 
t\in ((n-1)\tau,n\tau],\,n\geq 0:\qquad
\left\{
\begin{array}{l}
\tilde \rho^{\tau}(t) =\rho^{n+\frac 12},\\
\rho^{\tau}(t) =\rho^{n+1}.
\end{array}
\right.
\end{equation*}
By construction we have the energy monotonicity
\begin{equation*}
 \label{eq:energy_monotonicity}
 \forall\, 0\leq t_1\leq t_2:\qquad
\mathcal{F}(\rho^{\tau}(t_2))\leq \mathcal{F}(\tilde \rho^{\tau}(t_2))\leq \mathcal{F}(\rho^{\tau}(t_1))\leq \mathcal{F}(\tilde \rho^{\tau}(t_1))\leq \mathcal{F}(\rho^0),
\end{equation*}
and an easy application of the Cauchy-Schwarz inequality with the total square-distance estimate \eqref{eq:total_square_distance_d} gives moreover the classical $\frac 12$-H\"older estimate
\begin{equation}
 \label{eq:approximate_1/2_Holder_estimate}
 \forall\, 0\leq t_1\leq t_2:\qquad 
 \left\{
 \begin{array}{l}
 \d(\rho^\tau(t_2), \rho^\tau(t_1))\leq C|t_2-t_1+\tau|^{\frac{1}{2}}\\
 \d(\tilde\rho^\tau(t_2),\tilde\rho^\tau(t_1))\leq C|t_2-t_1+\tau|^{\frac{1}{2}}
 \end{array}
 \right. .
\end{equation}
Moreover for all $t> 0$ we have $\tilde\rho^\tau(t)=\rho^{n+\frac 12}$ and $\rho^\tau(t)=\rho^{n+1}$ for some $n\geq 0$. From the total square estimate \eqref{eq:total_square_distance_hybrid} we have therefore $\H^2(\tilde{\rho}^\tau(t),\rho^\tau(t))\leq C\tau$, and by Proposition~\ref{prop:comparison_d_W_H} we conclude that the two curves $\rho^\tau,\tilde\rho^\tau$ stay close
\begin{equation}
\label{eq:curves_d_close} 
 \forall\,t\geq 0:\qquad
 \d(\tilde\rho^\tau(t),\rho^\tau(t))\leq \H(\tilde\rho^\tau(t),\rho^\tau(t))\leq C\sqrt{\tau}
\end{equation}
uniformly in $\tau$.

As a fairly general consequence of the total-square distance estimate \eqref{eq:total_square_distance_d}, we retrieve an abstract convergence (pointwise in time) when $\tau\to 0$ for a weak topology:
\begin{cor}
\label{cor:CV_weak*_interpolation_to_curve}
Assume that $\mathcal{F}(\rho^0)<\infty$ and $\mathcal F$ is bounded from below on $\M^+$.
Then there exists a $\d$-continuous curve $\rho\in\mathcal C^{\frac{1}{2}}([0,\infty);\M^+_{\d})$ and a discrete subsequence $\tau\to 0$ (not relabeled here) such that
\begin{equation*}
\forall\,t\geq 0:\qquad
\rho^\tau(t),\tilde\rho^\tau(t)\to \rho(t)\quad\mbox{weakly-}\ast\mbox{ when }\tau\to 0.
\label{eq:pointwise_CV_weak*}
\end{equation*}
\end{cor}
Note that our statement is again unrelated to any PDE consideration, and merely exploits the metric structure.
We recall that the weak-$\ast$ convergence of measures is defined in duality with $\mathcal C_0(\Omega)$ test-functions.
Observe that the two interpolated curves converge to the \emph{same} limit, and note that because $\rho\in \mathcal C([0,\infty);\M^+_\d)$ the initial datum $\rho(0)=\rho^0$ is taken continuously in the $\KFR$ metric sense.
In particular since $\d$ metrizes the narrow convergence of measures \cite[thm. 3]{KMV_15} the initial datum $\rho(0)=\rho^0$ will be taken at least in the narrow sense, which is stronger than weak-$\ast$ or distributional convergence.
\begin{proof}
From the proof of \cite[lem. 2.2]{KMV_15} it is easy to see that we have mass control
$$
\forall\,\mu,\nu\in \M^+:\qquad |\nu|\leq |\mu|+\d^2(\nu,\mu).
$$
Applying this with $\nu=\rho^\tau(t),\tilde\rho^\tau(t)$ and $\mu=\rho^0$, and noting that the square-distance estimate \eqref{eq:total_square_distance_d} controls $\d^2(\rho^\tau(t),\rho^0),\d^2(\tilde\rho^\tau(t),\rho^0)\leq C(t+\tau)$, we see that the masses are controlled as $|\rho^\tau(t)|+|\tilde\rho^\tau(t)|\leq C(1+T)$ uniformly in $\tau$ in any finite time interval $t\in [0,T]$.
By the Banach-Alaoglu in $\M=\mathcal C_0^\ast$ we see that $\rho^\tau(t),\tilde\rho^\tau(t)$ lie in the fixed weakly-$\ast$ relatively compact set $\mathcal K_T=\{|\rho|\leq C(1+T)\}$ for all $t\in [0,T]$.
By \cite[thm. 5]{KMV_15} we know that the $\KFR$ distance is lower semi-continuous with respect to the weak-$\ast$ convergence of measures, and the metric space $(\M^+,\d)$ is complete \cite[thm. 3]{KMV_15}.
Exploiting the time equicontinuity \eqref{eq:approximate_1/2_Holder_estimate}, the lower semi-continuity, and the completeness, we can apply a refined version of the Arzel\`a-Ascoli theorem \cite[prop. 3.3.1]{AGS_08} to conclude that, up to extraction of a discrete subsequence if needed, $\rho^\tau(t)\to \rho(t)$ and $\tilde\rho^\tau(t)\to \tilde\rho(t)$ pointwise in $t\in [0,T]$ for the weak-$\ast$ convergence and for some limit curves $\rho,\tilde\rho\in \mathcal C^{\frac{1}{2}}([0,T];\M^+_{\d})$. 
Moreover $\rho(t),\tilde\rho(t)\in \mathcal K_T$ for all $t\in [0,T]$, and by diagonal extraction we can assume that this holds for all $T>0$.
Finally as we already know that $\rho^\tau(t)$ and $\tilde\rho^\tau(t)$ converge weakly-$\ast$ to $\rho(t)$ and $\tilde\rho(t)$ respectively, we conclude by \eqref{eq:curves_d_close} and lower semi-continuity that $\d(\rho(t),\tilde\rho(t))\leq \liminf\limits_{\tau\to 0} \d(\rho^\tau(t),\tilde\rho^\tau(t))=0$ for any arbitrary $t\geq 0$.
 Thus $\rho=\tilde\rho$ as desired and the proof is complete.
\end{proof}
}

{ In order to connect now the previous abstract metric considerations with the PDE framework, we detail each of the substeps \eqref{eq:step_wasserstein}\eqref{eq:step_hellinger_fisher_rao} and exploit the particular $\MK,\FR$ Riemannian structures to retrieve the corresponding Euler-Lagrange equations.

In order to keep our notations light we write $\mu$ for the previous step and $\rho^*$ for the minimizer.
Thus $\mu=\rho^n$ and $\rho^*=\rho^{n+\frac 12}$ in the first $\MK$ step $\rho^n\to\rho^{n+\frac 12}$, while $\mu=\rho^{n+\frac 12}$ and $\rho^*=\rho^{n+1}$ in the next $\FR$ step $\rho^{n+\frac 12}\to \rho^{n+1}$.}
 %
 \subsection{The Monge-Kantorovich substep}
 \label{section:Wasserstein_substep}
 For some fixed absolutely continuous measure $\mu\in \M_2^+$ (finite second moment) and mass $|\mu|=m$, let us consider here an elementary minimization step
 \begin{equation}
 \rho^* \in \underset{\rho\in\M^+_2,\,|\rho|=m}{\operatorname{Argmin}}\left\{\frac{1}{2\tau}\W^2(\rho,\mu)+\mathcal F(\rho)\right\}.
 \label{eq:wasserstein_step_general}
 \end{equation}
 Note that, if $\Omega$ is bounded, the restriction on finite second moments can be relaxed.
 Further assuming that $\mathcal F$ is lower semi-continuous with respect to the weak $L^1$ convergence (which is typically satisfied for the classical models), it is easy to obtain an absolutely continuous minimizer $\rho^*\in\M^+_2$ with mass $|\rho^*|=m=|\mu|$.
 Additional assumptions (e.g. strict convexity) sometimes guarantee uniqueness.
 Here we do not take interest in optimal conditions guaranteeing existence and/or uniqueness of minimizers, and this should be checked on a case-to-case basis depending on the structure of $U,\Psi,K$.
 
 From the classical theory of optimal transportation we know that there exists a (backward) optimal map $\mathbf t$ from $\rho^*$ to $\mu$, such that
 $$
 \W^2(\rho^*,\mu)=\int_\Omega\left|x-\mathbf t(x)\right|^2\rd \rho^*(x).
 $$
 A by-now standard computation \cite{Santambroggio_book,villani_small} shows that the Euler-Lagrange equation associated with \eqref{eq:wasserstein_step_general} can be written in the form
\begin{equation}
\label{eq:euler_lagrange_wasserstein_riemannian}
\forall \,\mathbf \zeta\in\mathcal{C}^\infty_c(\Omega;\R^d):\qquad
\int _\Omega \frac{\operatorname{id}-\mathbf{t}}{\tau}\cdot \mathbf \zeta\,\rd \rho^*+\int_\Omega \nabla (U'(\rho^*)+\Psi+K\star\rho^*)\cdot \mathbf \zeta\, \rd\rho^*=0.
 \end{equation}
 Using the definition of the pushforward $\mu=\mathbf t\#\rho^*$ we recall the classical Taylor expansion
 \begin{multline}
  \label{eq:Taylor_expansion_W2}
\int_\Omega (\rho^*-\mu)\phi  =\int_{\Omega}(\rho^*-\mathbf t\#\rho^*)\phi = \int_\Omega \big(\phi(x)-\phi(\mathbf t(x))\big)\rho^*(x)\\
 =\int_\Omega\Big(x-\mathbf t(x))\cdot \nabla \phi(x)+\mathcal O\left(\|D^2\phi\|_\infty|x-\mathbf t(x)|^2\right)\Big)\,\rd\rho^*(x) \\
 =\int_\Omega(\operatorname{id}-\mathbf{t})\cdot \nabla \phi\,\rd\rho^* + \mathcal{O}\left(\|D^2\phi\|_\infty\W^2(\rho^*,\mu)\right)
 \end{multline}
 for all $\phi\in\mathcal{C}^\infty_c(\Omega)$.
 Taking $\mathbf \zeta=\nabla\phi$ in \eqref{eq:euler_lagrange_wasserstein_riemannian} and substituting finally yields
 \begin{equation}
\int_\Omega (\rho^*-\mu)\phi=-\tau\int_\Omega \nabla( U'(\rho^*)+\Psi+ K\star \rho^*)\cdot \nabla\phi\, \rd\rho^*
+ \mathcal{O}\left(\|D^2\phi\|_\infty\W^2(\rho^*,\mu)\right)
 \label{eq:euler_lagrange_wasserstein_general}
 \end{equation}
 for all smooth test functions $\phi$.
 This is of course an approximation of the implicit implicit Euler scheme
 $$
 \frac{\rho^*-\mu}{\tau}=\dive(\rho^*\nabla (U'(\rho^*+\Psi+ K\star \rho^*)),
 $$
 the approximate error being controlled quadratically in the $\W$ distance. Note that this corresponds to the pure transport part $\partial_t\rho =\dive(\rho\nabla (U'(\rho)+\Psi+ K\star \rho^*))+(\ldots)$ 
 \tg{in the PDE 
 \eqref{eq:PDE}}.
 
 \subsection{The Fisher-Rao substep}
 \label{section:hellinger_substep}
Let us fix as before an arbitrary measure $\mu\in \M^+$ (no restriction on the second moment), and assume that there exists somehow an absolutely continuous minimizer
 \begin{equation}
 \rho^* \in \underset{\rho\in\M^+}{\operatorname{Argmin}}\left\{\frac{1}{2\tau}\H^2(\rho,\mu)+\mathcal F(\rho)\right\}.
 \label{eq:FRH_step_general}
 \end{equation}
The existence and uniqueness of minimizers can again be obtained under suitable superlinearity, lower semi-continuity, and convexity assumptions on $U,\Psi,K$, and we do not worry about this issue.

Let us start by differentiating the squared distance for suitable perturbations $\rho_\eps$ of the minimizer $\rho^*$.
According to section~\ref{section:HFR_metrics} an arbitrary $\psi\in\mathcal C^\infty_c(\Omega)$ \tg{is identified to }
a tangent vector in $T_{\rho^*}\M^+_\H$ through
$$
\left\{
\begin{array}{l}
 \partial_\eps \rho_\eps =\rho_\eps\psi\\
 \rho_{\eps}(0)=\rho^*
\end{array}
\right.
\qquad \Leftrightarrow \qquad \rho_\eps=\rho^*e^{\eps\psi}.
$$
Denoting by $\mu_s=[(1-s)\sqrt{\mu}+s\sqrt{\rho^*}]^2$ the Fisher-Rao geodesics from $\mu$ to $\rho^*$, the terminal velocity $\partial_s\mu(1)=2\sqrt{\rho^*}(\sqrt{\rho^*}-\sqrt{\mu})$ can be represented by the $L^2(\rd\rho^*)$ action of $r=2\frac{\sqrt{\rho^*}-\sqrt{\mu}}{\sqrt{\rho^*}}$.
Using the first variation formula $\left.\frac{d}{dt}\left(\frac{1}{2}d^2(x(t),y)\right)\right|_{t=0}=\left<x'(0),y'(1)\right>_{x(0)}$ and our $L^2(\rd\rho)$ identification of the tangent spaces in section~\ref{section:FRHKW} we can guess that
\begin{align*}
\left.\frac{d}{d\eps}\left(\frac 1 2 \H^2(\rho_\eps,\mu)\right)\right|_{\eps=0} &=\left<\partial_\eps\rho(0),\partial_s\mu(1)\right>_{T_{\rho^*}\M^+_\H}\\
&=\left(\psi,r\right)_{L^2(\rd\rho^*)}=2\int_{\Omega} (\sqrt{\rho^*}-\sqrt{\mu})\sqrt{\rho^*}\psi,
\end{align*}
which can be checked by differentiating w.r.t. $\eps$ in the explicit representation \eqref{eq:def_HFR}.
Using the same Riemannian formalism we similarly anticipate that
\begin{align*}
\left.\frac{d}{d\eps}\mathcal F(\rho_\eps)\right|_{\eps=0} & = \left<\grad_\H\mathcal F,\partial_\eps\rho(0)\right> _{T_{\rho^*}\M^+_\H}\\
& = \left<F'(\rho^*),\psi\right>_{L^2(\rd\rho^*)} =\int_\Omega\rho^*(U'(\rho^*)+\Psi+ K\star \rho^*)\psi,
\end{align*}
and this can be checked again by differentiating $\frac{d}{d\eps}\mathcal F(\rho_\eps)=\int_\Omega\partial_\eps(\ldots)$ under the integral sign.
Writing the the optimality condition $\left.\frac{d}{d\eps}\left(\frac{1}{2\tau}\H^2(\rho_\eps,\mu)+\mathcal F(\rho_\eps)\right)\right|_{\eps=0}=0$ thus gives the Euler-Lagrange equation
\begin{equation}
\forall \,\psi\in\mathcal C^\infty_c(\Omega):\qquad
\int_{\Omega} (\sqrt{\rho^*}-\sqrt{\mu})\sqrt{\rho^*}\psi=-\frac{\tau}{2}\int_\Omega \big\{U'(\rho^*)+\Psi + K\star \rho^*\big\}\rho^*\psi.
\label{eq:euler_lagrange_hellinger_general}
\end{equation}
In order to relate this with the more standard Euclidean difference quotient, we first assume that  $U'(\rho^*)+\Psi + K\star\rho^*\in L^2(\rd\rho^*)$, or in other words that $\grad_\H\mathcal F(\rho^*)$ can indeed be considered as a tangent vector of $T_{\rho^*}\mathcal M^+_\H$.
This should be natural, but may require a case-to-case analysis depending on the structure of $U,\Psi,K$.
\tg{Then} an easy density argument shows that the previous equality holds for all $\psi\in L^2(\rd\rho^*)$.
Taking in particular $\psi=\frac{\sqrt{\rho^*}+\sqrt{\mu}}{\sqrt{\rho^*}}\phi\in L^2(\rd\rho^*)$ for arbitrary $\phi\in \mathcal C^\infty_c(\Omega)$, we obtain a slight variant of the previous Euler-Lagrange equation \eqref{eq:euler_lagrange_hellinger_general} in the form
\begin{equation}
\label{eq:euler_lagrange_hellinger_general_twisted}
\forall \,\phi\in\mathcal C^\infty_c(\Omega):\qquad
\int_{\Omega}(\rho^*-\mu)\phi=-\tau\int_\Omega \frac{\sqrt{\rho^*}(\sqrt{\rho^*}+\sqrt{\mu})}{2}\big\{U'(\rho^*)+\Psi+ K\star \rho^*\big\}\phi.
\end{equation}
Recalling that in the minimizing scheme we only deal with measures at short $\mathcal{O}(\sqrt{\tau})$ distance, one should essentially think of this as if $\rho^*\approx \mu$ in the right-hand side, and \eqref{eq:euler_lagrange_hellinger_general_twisted} is thus an approximation of the implicit Euler scheme
$$
\frac{\rho^*-\mu}{\tau}=-\rho^*(U'(\rho^*)+\Psi+K\star \rho^*).
$$
Note that this is the reaction part $\partial_t\rho=(\ldots)-\rho(U'(\rho)+\Psi+K\star \rho^*)$ in the PDE \eqref{eq:PDE}.
\begin{rmk}
\label{rmk:no_asymptotic_expansion_Hellinger}

Contrarily to the corresponding approximate Euler-Lagrange equation \eqref{eq:euler_lagrange_wasserstein_general} for one elementary Monge-Kantorovich substep, \eqref{eq:euler_lagrange_hellinger_general_twisted} does not involve any quadratic remainder $\mathcal{O}(\H^2(\rho^*,\mu))$.
The price to pay for this is that the right-hand side appears now as a slight ``twist'' of the more natural and purely Riemannian object $-\rho^*(U'(\rho^*)+\Psi)=-\grad_\H \mathcal F(\rho^*)$ in \eqref{eq:euler_lagrange_hellinger_general_twisted}, the twist occurring through the approximation $\frac{\sqrt{\rho^*}(\sqrt{\rho^*}+\sqrt{\mu})}{2}\approx \rho^*$.
\end{rmk}
\begin{rmk}
A technical issue might arise here for unbounded domains.
Indeed since we construct recursively $\rho^n\overset{ \W^2}{\longrightarrow }\rho^{n+\frac{1}{2}}\overset{\H^2}{\longrightarrow } \rho^{n+1}$ one should make sure that, in the second reaction substep, the minimizer $\rho^{n+1}$ keeps finite second moment so that the scheme can be safely iterated afterward.
This should be generally guaranteed if the external potential $\Psi$ is quadratically confining, but may require once again a delicate analysis depending on the structure of $U,\Psi,K$ (we will show in section~\ref{section:compactness} that this holds e.g. in the simple case $\Psi,K\equiv 0$).
\end{rmk}

\subsection{Convergence to a weak solution}
\label{section:CV_to_weak_solution}
~
{ 
We can now show that, under some strong compactness assumptions, the limit $\rho=\lim \rho^\tau=\lim\tilde\rho^\tau$ is generically a weak solution to the original PDE.
\begin{theo}
\label{theo:CV_to_weak_solution}
 Let $\rho^\tau,\tilde\rho^\tau,\rho$ as in Corollary ~\ref{cor:CV_weak*_interpolation_to_curve}, and assume that
 \begin{equation}
 \label{eq:assumption_strong_CV}
  \left\{
  \begin{array}{rcl}
   \tilde{\rho}^\tau\nabla\left(U'(\tilde\rho^\tau)+\Psi + K\star \tilde\rho^\tau\right)&\rightharpoonup &\rho\nabla\left(U'(\rho)+\Psi+K\star \rho\right)\\
   \sqrt{\rho^\tau}\frac{\sqrt{\rho^\tau}+\sqrt{\tilde\rho^\tau}}{2} (U'(\rho^\tau)+\Psi+K\star \rho^\tau) &\rightharpoonup &\rho(U'(\rho)+\Psi+K\star \rho)
  \end{array}
\right.
 \end{equation}
 at least weakly in $L^1_{\mathrm{loc}}((0,\infty)\times\Omega)$.
Then $\rho$ is a nonnegative weak solution of
\begin{equation*}
\label{eq:IBV_problem_full} 
\left\{
\begin{array}{ll}
 \partial_t \rho=\dive(\rho\nabla(U'(\rho)+\Psi+K\star\rho))-\rho(F'(\rho)+\Psi+K\star\rho) & \mbox{in }(0,\infty)\times \Omega\\
 \rho_{|t=0} = \rho^0 & \mbox{in }\M^+(\Omega)
\end{array}
\right.
\end{equation*}
\end{theo}
\noindent For the sake of generality we simply assumed here that the nonlinear terms pass to the limit as in \eqref{eq:assumption_strong_CV}.
This is of course a strong hypothesis to be checked in each case of interest, and usually requires \emph{strong convergence} $\rho^\tau,\tilde\rho^\tau\to\rho$ (e.g. pointwise a.e.).
We shall discuss in section~\ref{section:compactness} some strategies to retrieve such compactness.}
\begin{proof}
As already discussed after Corollary~\ref{cor:CV_weak*_interpolation_to_curve}, the initial datum $\rho(0)=\rho^0$ is taken continuously at least in the metric sense $(\M^+,\d)$. 
Moreover, any limit $\rho=\lim\limits_{\tau\to 0} \rho^\tau$ in any weak sense will automatically be nonnegative.

Fix now any $0<t_1<t_2$ and $\phi\in \mathcal{C}^\infty_c(\Omega)$. 
For fixed $\tau$ we have $\rho^\tau(t_i)=\rho^{N_i}$ for $N_i=\lceil t_i/\tau\rceil$, and $T_i=N_i\tau\to t_i$ as $\tau\to 0$.
Moreover for fixed $n\geq 0$ we have by construction the two Euler-Lagrange equations \eqref{eq:euler_lagrange_wasserstein_general}\eqref{eq:euler_lagrange_hellinger_general_twisted}, one for each Monge-Kantorovich and Fisher-Rao substep as in section~\ref{section:Wasserstein_substep} and section~\ref{section:hellinger_substep} respectively. More explicitly, there holds
\begin{multline*}
\int_\Omega (\rho^{n+\frac 12}-\rho^n)\phi=-\tau\int_\Omega \rho^{n+\frac 1 2}\nabla( U'(\rho^{n+\frac 12})+\Psi+ K\star \rho^{n+\frac 12})\cdot \nabla\phi\\
+ \mathcal{O}\left(\|D^2\phi\|_\infty\W^2(\rho^{n+\frac 12},\rho^n)\right)
\end{multline*}
and
$$
\int_{\Omega}(\rho^{n+1}-\rho^{n+\frac 12})\phi=-\tau\int_\Omega \frac{\sqrt{\rho^{n+1}}(\sqrt{\rho^{n+1}}+\sqrt{\rho^{n+\frac 12}})}{2}\big\{U'(\rho^{n+1})+\Psi + K\star \rho^{n+1} \big\}\phi.
$$
Summing from $n=N_1$ to $n=N_2-1$, using the square-distance estimate \eqref{eq:total_square_distance_hybrid} to control the remainder term in the first Euler-Lagrange equation above, and recalling that the interpolated curves are piecewise constant, we immediately get
\begin{multline*}
\int_\Omega \left(\rho^\tau(t_2)-\rho^\tau(t_1)\right)\phi  = \sum\limits_{n=N_1}^{N_2-1}\int_\Omega \left\{(\rho^{n+1}-\rho^{n+\frac 12})+(\rho^{n+\frac 12}-\rho^n)\right\}\phi\\
{ 
=-\sum\limits_{n=N_1}^{N_2-1}\tau \int_\Omega \frac{\sqrt{\rho^{n+1}}(\sqrt{\rho^{n+1}}+\sqrt{\rho^{n+\frac 12}})}{2}\big\{U'(\rho^{n+1})+\Psi + K\star \rho^{n+1} \big\}\phi
}\\
{ 
-\sum\limits_{n=N_1}^{N_2-1}\tau\int_\Omega \rho^{n+\frac 1 2}\nabla( U'(\rho^{n+\frac 12})+\Psi+ K\star \rho^{n+\frac 12})\cdot \nabla\phi}\\
{ 
+ \mathcal O\left(\|D^2\phi\|_\infty\sum\limits_{n=N_1}^{N_2-1}\W^2(\rho^{n+\frac 12},\rho^n)\right)}
\\
 = -\int_{T_1}^{T_2}\int_\Omega \frac{\sqrt{\rho^{\tau}}(\sqrt{\rho^{\tau}}+\sqrt{\tilde\rho^{\tau}})}{2}\big\{U'(\rho^{\tau})+\Psi + K \star \rho^{\tau}\big\}\phi\\
   -\int_{T_1}^{T_2}\int_\Omega \tilde\rho^{\tau}\nabla( U'(\tilde\rho^{\tau})+\Psi + K \star \tilde\rho ^\tau)\cdot \nabla\phi \quad +\mathcal O\left(\|D^2\phi\|_\infty\tau\right).
\end{multline*}
From Corollary~\ref{cor:CV_weak*_interpolation_to_curve} we know that $\rho^\tau(t)$ converge weakly-$\ast$ to $\rho(t)$ pointwise in time, so the left-hand side passes to the limit when $\tau\to 0$. Due to our strong assumption \eqref{eq:assumption_strong_CV} and because $T_i\to t_i$ the right-hand side also passes to the limit.
As a consequence we get
\begin{equation*}
\label{eq:weak_formulation}
\int_\Omega \left(\rho(t_2)-\rho(t_1)\right)\phi  = -\int_{t_1}^{t_2}\int_\Omega \rho\Big(\nabla( U'(\rho)+\Psi + K \star \rho)\cdot \nabla\phi + \big(U'(\rho)+\Psi + K \star \rho\big)\phi\Big)
\end{equation*}
for all $0<t_1<t_2$ and $\phi\in\mathcal C^\infty_c(\Omega)$, which is clearly an admissible weak formulation of $\partial_t\rho=\dive(\rho\nabla(U'(\rho)+\Psi + K \star \rho))-\rho(U'(\rho)+\Psi + K \star \rho)$.
\end{proof}

 If $\Omega\neq \R^d$ some further work may be needed to retrieve the homogeneous Neumann condition $\rho\nabla(U'(\rho)+\Psi+ K \star \rho)\cdot \nu=0$ on $\partial\Omega$.
 This amounts to extending the class of $\mathcal C^\infty_c(\Omega)$ test functions to $\mathcal{C}^1_{loc}(\overline\Omega)$ and should generically hold with just enough regularity on the solution, but we will disregard this technical issue for the sake of simplicity.

\section{Compactness issues: an illustrative example}
\label{section:compactness}
In Theorem~\ref{theo:CV_to_weak_solution} we assumed for simplicity that the nonlinear terms pass to the limit, mainly in the distributional sense. 
In order to prove this, the usual strategy is to obtain first some energy/dissipation-type estimates to show that the nonlinear terms have a weak limit, and then prove pointwise convergence $\rho^\tau(t,x)\to \rho(t,x)$ a.e. $(t,x)\in \R^+\times\Omega$ to identify the weak limit (typically as weak-strong products of limits). 
Thus the problem should amount to retrieving enough compactness on the interpolating curves $\rho^\tau,\tilde\rho^\tau$. 
With the help of any Aubin-Lions-Simon type results this essentially requires compactness in time and space, which can be handled separately for different topologies in a flexible way. 
Compactness in space usually follows from the aforementioned energy/dissipation estimates, and the energy monotonicity should of course help: if e.g. the total energy $\mathcal F(\rho)=\int_\Omega U(\rho)+(\ldots)$ controls any $L^q(\Omega)$ norm then $\mathcal{F}(\rho^\tau(t))\leq \mathcal F(\rho^0)$ immediately controls $\|\rho^\tau\|_{L^\infty(0,\infty;L^q)}$ uniformly in $\tau$.
A rule of thumbs for parabolic equations is usually that space regularity can be transferred to time regularity.
Thus the parabolic nature of the scheme should allow here to transfer space estimates, if any, to time estimates. Note also that some sort of time compactness (approximate equicontinuity) is already guaranteed by \eqref{eq:approximate_1/2_Holder_estimate}, but in a very weak metric sense for which the standard Aubin-Lions-Simon theory does not apply directly.

A slight modification of the usual arguments should however be required here, because the scheme is decomposed in two separate substeps. 
The first Monge-Kantorovich substep \eqref{eq:wasserstein_step_general} encodes the higher order part of the PDE, which is parabolic and should therefore be smoothing.
This regularization can often be quantified using by-now classical methods in (Monge-Kantorovich) optimal transport theory, such as BV estimates \cite{Santambroggio_BV_15}, the \emph{flow-interchange} technique from \cite{matthes_interchange_09}, regularizing $\lambda$-displacement convexity in the spirit of \cite{AGS_08,mccann_convexity_97}, or any other strategy.
On the other hand the second Fisher-Rao substep \eqref{eq:FRH_step_general} encodes the reaction part of the PDE, hence we cannot expect any smoothing at this stage.
One should therefore make sure that, in the step $\rho^{n+\frac 12}\overset{\H}{\longrightarrow}\rho^{n+1}$, the regularity of $\rho^{n+\frac 12}$ inherited from the previous step $\rho^{n}\overset{\W}{\longrightarrow}\rho^{n+\frac 12}$ propagates to $\rho^{n+1}$ at least to some extent.

At this stage we would like to point out one other possible advantage of our splitting scheme: it is well known \cite{AGS_08} that $\lambda$-geodesic convexity is a central tool in the theory of gradient flows in abstract metric spaces, and leads to quantified regularization properties at the discrete level. 
Second order differential calculus {\it \`a la Otto} \cite{Otto_PME_2001} with respect to the $\KFR$ Riemannian structure was discussed in \cite{KMV_15,LMS_small_15} (also earlier suggested in \cite{liero2013gradient}) and allows to determine at least formally if a given functional $\mathcal F$ is $\lambda$-geodesically convex for the distance $\d$.
However, in our scheme each step only sees either one of the differential $\W/\H$ structures and therefore only separate geodesic convexity comes into play.
Consider for example the case of internal energies $\mathcal F(\rho)=\int_\Omega U(\rho)$.
Then the celebrated condition for McCann's displacement convexity \cite{mccann_convexity_97} with respect to $\W$ reads $\rho P'(\rho)-\left(1-\frac 1d\right)P(\rho)\geq 0$ in space dimension $d$, where the pressure $P(\rho):=\rho U'(\rho)-U(\rho)$.
On the other hand using the Riemannian formalism in section~\ref{section:HFR_metrics} it is easy to see that, at least formally, this same functional is $\lambda$-geodesically convex with respect to $\H$ if and only if $\rho U''(\rho)+\frac{U'(\rho)}{2}\geq \lambda$.
This condition can be interpreted as $s\mapsto U(s^2)$ being $\lambda/4$-convex in $s=\sqrt{\rho}$, the latter change of variables naturally arising through \eqref{eq:def_HFR} and $\FR^2(\rho_0,\rho_1)=4\|\sqrt{\rho_1}-\sqrt{\rho_0}\|^2_{L^2}$.
Those two conditions are very easy to check separately and, in the light of the infinitesimal uncoupling, it seems likely that simultaneous convexity with respect to each of the $\W,\H$ metrics is equivalent to convexity with respect to the coupled $\KFR$ structure.
See \cite[section 3]{KMV_15} and \cite[section 5.1]{LMS_small_15} for related discussions.\\

The rest of this section is devoted to the illustration of possible compactness strategies in the simple case
\begin{equation}
\label{hyp:structure_U_Psi}
\left\{
\begin{array}{l}
 \Psi,K\equiv 0,\\
 U\in\mathcal C^1([0,\infty))\cap \mathcal C^2(0,\infty)
 \quad\mbox{ with }U(0)=0,\\
 U',U''\geq 0,\\
 \rho U''(\rho)
 \quad \mbox{ is bounded for small }\rho\in (0,\rho_0],
\end{array}
\right.
\tag{H}
\end{equation}
which from now will be assumed without further mention.
We would like to stress here that \eqref{hyp:structure_U_Psi} holds for any Porous-Medium-type nonlinearity $U_m(\rho)=\frac 1{m-1}\rho^m$ at least in the slow diffusion regime $m>1$, but \emph{not} for the Boltzmann entropy $H(\rho)=\rho\log\rho-\rho$.
Even though the latter is well behaved (displacement convex) with respect to the Monge-Kantorovich structure \cite{JKO_98,villani_small}, it is \emph{not} with respect to the Fisher-Rao one. Indeed it is easy to check that $H(\rho)$ is not convex in $\sqrt{\rho}$, so that the Boltzmann entropy is not $\lambda$-displacement convex with respect to $\H$ for any $\lambda\in\R$.
This would require $\rho H''(\rho)+\frac{H'(\rho)}{2}=1+\frac{\log \rho}{2}\geq\lambda$ for some constant $\lambda$, which obviously fails for small $\rho$ (this can be related to $\rho=0$ being an extremal point in $\M^+$, where all the Riemannian formalism from section~\ref{section:FRHKW} degenerates).
{ Since the purpose of this section is to illustrate that strong compactness can be retrieved at least in some particular cases, we chose to set $\Psi\equiv 0$ to make the computations and estimates as light as possible.
The case $\Psi\not\equiv 0$ follows with only minor modifications at least for reasonable potentials (see e.g. remark~\ref{rmk:implicit_functions_with potential} and \cite{theseMax,GLM_16}).
Including interaction terms $K\not\equiv 0$ may be more involved and require additional assumptions, and we shall not comment further on this.}
%
\subsection{Propagation of regularity at the discrete level}
\label{section:propagation_regularity}
Whenever $U',U''\geq 0$, the PDE $\partial_t \rho=\dive(\rho\nabla U'(\rho))-\rho U'(\rho)=\dive(\rho U''(\rho)\nabla \rho)-\rho U'(\rho)$ is formally parabolic, satisfies the maximum principle $\|\rho(t)\|_\infty\leq \|\rho^0\|_\infty$, and initial regularity should propagate.
We prove below that this holds at the discrete level:
\begin{prop}[BV and $L^\infty$ estimates]
\label{prop:propagation_BV}
Assume that the initial datum $\rho^0\in BV\cap L^\infty(\Omega)$. Then for any $\tau< 2/U'(\|\rho^0\|_\infty)$ there holds
$$
\forall\,t\geq 0:\qquad \|\rho^\tau(t)\|_{\mathrm{BV}(\Omega)}\leq \|\tilde\rho^\tau(t)\|_{\mathrm{BV}(\Omega)}\leq \|\rho^0\|_{\mathrm{BV}(\Omega)}
$$
and
$$
\forall\,t\geq 0:\qquad \|\rho^\tau(t)\|_{L^\infty(\Omega)}\leq \|\tilde\rho^\tau(t)\|_{L^\infty(\Omega)}\leq \|\rho^0\|_{L^\infty(\Omega)}.
$$
\end{prop}
\begin{proof}
We argue at the discrete level by showing that the estimates propagate in each substep.
We shall actually prove a more precise result, namely
\begin{equation}
\label{eq:propagation_regularity_Wasserstein}
\|\rho^{n+\frac{1}{2}}\|_{\mathrm{BV}} 	\leq 	\|\rho^{n}\|_{\mathrm{BV}},
\qquad
\|\rho^{n+\frac 12}\|_{L^{\infty}}	 \leq 	 \|\rho^{n}\|_{L^\infty}
\end{equation}
and
\begin{equation}
\label{eq:propagation_regularity_Hellinger}
\|\rho^{n+1}\|_{\mathrm{BV}} 	\leq 	\|\rho^{n+\frac{1}{2}}\|_{\mathrm{BV}},
\qquad
\rho^{n+ 1}(x)	 \leq 	 \rho^{n+\frac 12}(x)\mbox{ a.e.}
\end{equation}

The propagation \eqref{eq:propagation_regularity_Wasserstein} in the first $\MK$ step only requires convexity $U''\geq 0$ and no smallness condition on the time step $\tau$.
This should be expected since the $\MK$ step is an \emph{implicit} discretization of $\partial_t\rho=\dive(\rho\nabla U'(\rho))=\dive(\rho U''(\rho)\nabla\rho)$, which is formally parabolic as soon as $U''\geq 0$.
We recall first that by construction the step is mass preserving, $\|\rho^{n+\frac{1}{2}}\|_{L^1}=\|\rho^{n}\|_{L^1}$.
With our assumption $U''\geq 0$ we can directly apply \cite[thm. 1.1]{Santambroggio_BV_15} to obtain $\|\rho^{n+\frac{1}{2}}\|_{\mathrm{TV}}\leq \|\rho^{n}\|_{\mathrm{TV}}$, which immediately entails the BV estimate.
An early proof of $\|\rho^{n+\frac 12}\|_{L^\infty}	\leq 	\|\rho^{n}\|_{L^\infty}$ can be found in \cite{Otto_labyrinth_98} for the particular case $U(\rho)=\rho^2$, and the case of general convex $U$ is covered by \cite[prop. 7.32]{Santambroggio_book} (see also \cite{carlier_santambroggio_urban_planning_05,santambrogio_transport_concentration_07}).
For the propagation \eqref{eq:propagation_regularity_Hellinger} in the $\FR$ step we show below that the minimizer $\rho^{n+1}$ can be written as
$$
\rho^{n+1}(x) =R(\rho^{n+\frac 12}(x))
\qquad \mbox{a.e. }x\in\Omega
$$
for some $1$-Lipschitz function $R:\R^+\to \R^+$ with $R(0)=0$.
This will ensure that $0\leq \rho^{n+1}(x)\leq \rho^{n+\frac 12}(x)$ and entail the $L^\infty$ and $L^1$ bounds as well as the total variation estimate (see \cite{AFP_BV_00} for the composition of $\mathrm{Lip}\,\circ\,\mathrm{BV}$ maps). 
Note that $\rho^{n+1}(x)\leq \rho^{n+\frac 12}(x)$ shows in particular that the second moments propagate to the next step, which should require further assumptions on $U,\Psi$ in the general case.
In the rest of the proof we write $\rho^*=\rho^{n+1}$ and $\mu=\rho^{n+\frac 12}$ for simplicity, in agreement with the notations in section~\ref{section:hellinger_substep}.

By \eqref{eq:euler_lagrange_hellinger_general} with $\Psi,K\equiv 0$ we see that
\begin{equation}
\label{eq:implicit_rho_mu_bad}
(\sqrt{\rho^*}-\sqrt{\mu})\sqrt{\rho^*}=-\frac{\tau}{2}\rho^*U'(\rho^*)
\end{equation}
at least in $L^1_{\mathrm{loc}}(\Omega)$, hence a.e. $x\in\Omega$.
From $U'\geq 0$ we immediately get that either $\rho^*=0$ or $\sqrt{\rho^*}\leq \sqrt{\mu}$, which gives in any case $\rho^*(x)\leq \mu(x)$ a.e.\\
We show now that if the CFL condition $\tau\leq U'(\|\rho^0\|_\infty)/2$ holds then $\rho^*$ and $\mu$ share the same support, i-e $\rho^*(x)>0\Leftrightarrow \mu(x)>0$.
From the previous inequality $\rho^*\leq \mu$ we only have to show that $\rho^*(x)>0$ as soon as $\mu(x)>0$.
Assume by contradiction that there is some subset $E\subset \Omega$ with positive Lebesgue measure such that $\rho^*(x)=0$ but $\mu(x)>0$ in $E$.
We claim that
$$
\overline\rho:=\rho^*\chi_{E^\complement}+\mu\chi_E
$$
is then a strictly better competitor than the minimizer $\rho^*$.
In order to check this we first compute the square distance
\begin{multline*}
\frac{1}{4}\left(\H^2(\overline\rho,\mu)-\H^2(\rho^*,\mu)\right)  =\int_\Omega\left|\sqrt{\overline\rho}-\sqrt{\mu}\right|^2 -\int_\Omega\left|\sqrt{\rho^*}-\sqrt{\mu}\right|^2\\
 =\left(\int_{E^\complement}\left|\sqrt{\rho^*}-\sqrt{\mu}\right|^2	+	\int_E\left|\sqrt{\mu}-\sqrt{\mu}\right|^2\right)\\
-\left(\int_{E^\complement}\left|\sqrt{\rho^*}-\sqrt{\mu}\right|^2+\int_E\left|0-\sqrt{\mu}\right|^2\right) = -\int_E\mu <0.
\end{multline*}
For the energy contribution we have by convexity
\begin{align*}
\mathcal{F}(\overline\rho)-\mathcal{F}(\rho^*)  = \int_\Omega U(\overline{\rho})-U(\rho^*) & \leq \int_\Omega U'(\overline\rho)(\overline{\rho}-\rho^*) \\
& = \int_E U'(\overline\rho) (\mu-0)\leq U'(\|\rho^0\|_\infty)\int_E\mu.
\end{align*}
Note that $0\leq \rho^*,\overline{\rho},\mu\leq \|\rho^0\|_\infty$ almost everywhere, so that all these integrals are well-defined.
Gathering these two inequalities we obtain
$$
 \frac{1}{2\tau} \left(\H^2(\overline{\rho},\mu)-\H^2(\rho^*,\mu)\right)+\left(\mathcal{F}(\overline\rho)-\mathcal F(\rho^*)\right)
  \leq \left(-\frac{2}{\tau}+U'(\|\rho^0\|_\infty)\right)\int_E\mu<0
$$
because $\int_E\mu>0$ and $\tau<2/U'(\|\rho^0\|_\infty)$.
This shows that $\overline\rho$ is a strictly better competitor and yields the desired contradiction, thus $\rho^*>0\Leftrightarrow \mu>0$.

Now inside the common support of $\rho^*,\mu$ we can divide \eqref{eq:implicit_rho_mu_bad} by $\sqrt{\rho^*}>0$, and $\rho=\rho^*(x)$ is a solution of the implicit equation
$$
f(\rho,\mu)=0
\quad\mbox{with}\quad
f(\rho,\mu):=\sqrt{\rho}\left(1+\frac{\tau}{2}U'(\rho)\right)-\sqrt{\mu}
$$
with $\mu=\mu(x)$ and a.e. $x\in \operatorname{supp}\rho^*=\operatorname{supp}\mu$.
An easy application of the implicit functions theorem shows that, for any $\mu>0$, this has a unique solution $\rho=R(\mu)$ for a $\mathcal C^1(0,\infty)$ function $R$ satisfying $0<R(\mu)\leq \mu$ for $\mu>0$.
Moreover one can compute explicitly for all $\mu>0$  
\begin{multline*}
0<\frac{dR}{d\mu}(\mu)=-\left.\frac{\partial_\mu f}{\partial_\rho f}\right|_{\rho=R(\mu)} =\frac{\frac{1}{2\sqrt{\mu}}}{\frac{1}{2\sqrt{\rho}}\left(1+\frac{\tau}{2}U'(\rho)\right)+\frac{
\tau}{2}\sqrt{\rho}U''(\rho)}\\
 \leq \frac{\frac{1}{2\sqrt{\mu}}}{\frac{1}{2\sqrt{\rho}}\left(1+\frac{\tau}{2}U'(\rho)\right)}=\frac{1}{\frac{\sqrt{\mu}}{\sqrt{\rho}}\left(1+\frac{\tau}{2}U'(\rho)\right)}=\frac{\rho}{\mu}\leq 1,
\end{multline*}
where we used successively $U''\geq 0$, $f(\rho,\mu)=0\Leftrightarrow 1+\frac{\tau}{2}U'(\rho) =\frac{\sqrt{\mu}}{\sqrt{\rho}}$, and $\rho=R(\mu)\leq \mu$.
Extending by continuity $R(0)=0$, we have shown that $\rho^*(x)= R(\mu(x))$ a.e. $x\in \Omega$ for some 1-Lipschitz function $R:\R^+\to\R^+$ with $R(0)=0$, and the proof is complete.
\end{proof}
\begin{rmk}
\label{rmk:implicit_functions_with potential}\tg{
A closer analysis of the implicit functions theorem above reveals that the argument only requires $U'\geq 0$ and $\rho U''(\rho)+U'(\rho)/2\geq 0$, which is less stringent than our assumption $U',U''\geq 0$ as in \eqref{hyp:structure_U_Psi}.
As already suggested this former condition corresponds to convexity of $s\mapsto U(s^2)$ in the $s=\sqrt{\rho}$ variable, or more intrinsically to geodesic convexity of $\mathcal F(\rho)=\int_\Omega U(\rho)$ with respect to the Fisher-Rao distance.
 We also point out that the same approach works with external potentials $\Psi\not\equiv 0$ under suitable structural assumptions: one shows first that strict positivity is preserved in the sense that $\operatorname{supp}\,\rho^{n+1}=\operatorname{supp}\,\rho^{n+\frac 1 2}$, which is to be expected since the ODE $\partial_t\rho =-\rho(U'(\rho)+\Psi(x))$ formally preserves positivity. 
 Exploiting the Euler-Lagrange equations \eqref{eq:euler_lagrange_hellinger_general}\eqref{eq:euler_lagrange_hellinger_general_twisted}, an implicit functions theorem $f(\rho,\mu,\Psi)=0\Leftrightarrow \rho =R(\mu,\Psi)$ then applies inside the common support to propagate the regularity.
 This still controls $\nabla \rho=\partial_\mu R\,\nabla \mu +\partial_\Psi R\nabla\Psi$ provided that $\Psi$ is smooth enough, see \cite{GLM_16,theseMax} for details.
}
\end{rmk}
~
\subsection{Compactness and Energy Dissipation Inequality}
\label{section:EDI}
~
{ 
In this section we check that our strong assumption \eqref{eq:assumption_strong_CV} in Theorem~\ref{theo:CV_to_weak_solution} holds in the particular case of internal energies only, i-e that the nonlinear terms in the PDE pass to the limit.
We start by improving the weak convergence in Corollary~\ref{cor:CV_weak*_interpolation_to_curve}:
}
\begin{prop}
\label{prop:strong_CV_L1_loc_L1}
Assume \eqref{hyp:structure_U_Psi}. Then
$$
\rho^\tau,\tilde\rho^\tau \to \rho
\qquad \mbox{in }L^1_{\mathrm{loc}}([0,\infty);L^1)
$$
for some (discrete) subsequence $\tau\to 0$.
\end{prop}
We give two proofs: the first one is elementary and fully exploits the uniform-in-time compactness estimates from Proposition~\ref{prop:propagation_BV}, {  which were derived here for the particular case $\Psi\equiv K\equiv 0$ only.
The second proof is less straightforward but enlightens the general idea of transferring space regularity to time regularity through the PDE itself, and should apply to non-trivial potentials and interactions with minor modifications}.
\begin{proof}[First proof of Proposition~\ref{prop:strong_CV_L1_loc_L1}]
Let us recall from Proposition~\ref{cor:CV_weak*_interpolation_to_curve} that $\rho^\tau(t),\tilde\rho^\tau(t)$ both converge weakly-$\ast$ to the same limit $\rho(t)$ pointwise in time. 
We claim that this weak-$\ast$ convergence can be strengthened into strong $L^1(\Omega)$ convergence.
Indeed for any fixed $t\geq 0$ we have $\|\rho^\tau(t)\|_{\mathrm{BV}},\|\tilde\rho^\tau(t)\|_{\mathrm{BV}}\leq \|\rho^0\|_{\mathrm{BV}}$ so by compactness $\mathrm{BV}(\Omega)\subset\subset L^1(\Omega)$ we see that $\{\rho^\tau(t)\}_{\tau>0},\{\tilde \rho^\tau(t)\}_{\tau>0}$ are $L^1$ relatively compact for fixed $t\geq 0$. 
Because strong $L^1$ convergence implies in particular weak-$\ast$ convergence of measures, and because we already know that these sequences are weakly-$\ast$ convergent, uniqueness of the limit shows in fact that the whole sequences are strongly converging in $L^1$ to the same limit 
$$
\forall\, t\geq 0:\qquad \lim\limits_{L^1}\rho^\tau(t)=\lim\limits_{\mathrm w -\ast} \rho^\tau(t)=\rho(t)=\lim\limits_{\mathrm w -\ast} \tilde\rho^\tau(t) = \lim\limits_{L^1}\tilde\rho^\tau(t).
$$
From this strong pointwise-in time $L^1$ convergence and the uniform $L^\infty(0,\infty;L^1)$ bounds from Proposition~\ref{prop:propagation_BV}, an easy application of Lebesgue's dominated convergence theorem in any finite time interval $[0,T]$ finally gives
strong $L^1((0,T);L^1)$ convergence for all $T>0$.
\end{proof}
Before giving the second proof we need a well known technical result:
\begin{lem}
\label{lem:wasserstein_controls_H-s}
Let $\mu_0,\mu_1$ be any absolutely continuous measures with finite second moments, same mass $|\mu_0|=|\mu_1|$, and bounded in $L^p(\Omega)$ for some $1\leq p\leq \infty$ by the same constant $C_p$. Then
$$
\forall \,\phi\in W^{1,2p'}(\Omega):\qquad \left|\int_\Omega(\mu_1-\mu_0)\phi\right|\leq \sqrt{C_p}\W(\mu_0,\mu_1)\|\nabla\phi\|_{L^{2p'}},
$$
with the convention $1'=\infty$ and $\infty'=1$.
\end{lem}

\begin{proof}
Let $(\mu_t,\v_t)_{t\in [0,1]}$ be the unique Monge-Kantorovich geodesics from $\mu_0$ to $\mu_1$, satisfying $\partial_t\mu_t+\dive(\mu_t\v_t)=0$ with constant metric speed $\|\v_t\|_{L^2(\rd\mu_t)}=cst=\W(\mu_0,\mu_1)$. 
We first claim that $\|\mu_t\|_{L^p}\leq C_p$ as well along this geodesics. 
Indeed if $p=1$ this is simply the mass conservation, and the proof for $p=\infty$ can be found in \cite{Otto_labyrinth_98}.
For $1<p<\infty$ this is a simple consequence of the displacement convexity of $\mathcal{E}_p[\mu]=\int_{\Omega}\frac{\mu^p}{p-1}$, \cite[thm. 5.15]{villani_small}.
Using the weak formulation of the continuity equation, we compute by H\"older's inequality
\begin{align*}
 \left|\int_\Omega(\mu_1-\mu_0)\phi\right| & =\left|\int_0^1\int_\Omega \v_t\cdot \nabla\phi \,\rd \mu_t\rd t\right|
 \leq \int_0^1\left(\int_\Omega|\v_t|^2\rd\mu_t\right)^{\frac{1}{2}}\left(\int_\Omega|\nabla\phi|^2\mu_t\right)^{\frac{1}{2}}\rd t\\
 & \leq \W(\mu_0,\mu_1)\int_0^1 (\|\mu_t\|_{L^p}\||\nabla\phi|^2\|_{L^{p'}})^{\frac 12}\rd t\leq \sqrt{C_p}\W(\mu_0,\mu_1)\|\nabla\phi\|_{L^{2p'}}
\end{align*}
and the proof is complete.
\end{proof}

\begin{proof}[Second proof of Proposition~\ref{prop:strong_CV_L1_loc_L1}]
Here we assume that $\Omega$ is bounded for simplicity, but the same argument would actually work for unbounded domains simply replacing all the functional spaces by their local counterparts ($\mathrm{BV}_{\mathrm{loc}}$, $H^1_{\mathrm{loc}}$, $L^p_{\mathrm{loc}}$\ldots).

We first control the difference quotient $\|\rho^{n+1}-\rho^{n}\|_Y$ in the dual space $Y:=H^1(\Omega)^*$.
For the Monge-Kantorovich step we can apply the previous Lemma~\ref{lem:wasserstein_controls_H-s} with $p=\infty$, $2p'=2$, $\|\rho^{n+\frac{1}{2}}\|_{L^\infty}\leq \|\rho^n\|_{L^\infty}\leq \|\rho^0\|_{L^\infty}$ and obtain by duality
$$
\|\rho^{n+\frac{1}{2}}-\rho^n\|_{Y}\leq C\,\W(\rho^{n+\frac{1}{2}},\rho^n).
$$
For the reaction step we recall the Euler-Lagrange equation \eqref{eq:euler_lagrange_hellinger_general_twisted}, which reads for $\Psi,K\equiv 0$
$$
\forall \,\phi\in\mathcal C^\infty_c(\Omega):\qquad
\int_{\Omega}(\rho^{n+1}-\rho^{n+\frac 12})\phi=-\tau\int_\Omega \frac{\sqrt{\rho^{n+1}}(\sqrt{\rho^{n+1}}+\sqrt{\rho^{n+\frac 12}})}{2}U'(\rho^{n+1})\phi.
$$
Because in the right-hand side $\rho^{n+\frac 12},\rho^{n+1}$ are bounded in $L^1\cap L^\infty(\Omega)$ uniformly in $n$ this gives 
$$
\|\rho^{n+1}-\rho^{n+\frac 12}\|_Y \leq \|\rho^{n+\frac 12}-\rho^{n+1}\|_{L^2}\leq C\tau.
$$
By triangular inequality we deduce from the previous two estimates that
$$
\|\rho^{n+1}-\rho^{n}\|_Y\leq C(\tau+\W(\rho^{n+1},\rho^n)),
$$
and using the square distance estimate \eqref{eq:total_square_distance_hybrid} and Cauchy-Schwarz inequality we obtain the approximate equicontinuity
$$
\forall\, 0\leq t_1\leq t_2:
\qquad
\|\rho^\tau(t_2)-\rho^\tau(t_1)\|_Y\leq C (|t_2-t_1+\tau|+|t_2-t_1+\tau|^{\frac{1}{2}}).
$$
Because the embedding $H^1\subset\subset L^2$ is compact we have $L^2\subset\subset Y$ as well.
Thanks to the $L^1\cap L^\infty(\Omega)$ bounds from Proposition~\ref{prop:propagation_BV} we have $\tau$-uniform bounds $ \|\rho^\tau(t)\|_{L^2}	\leq C$, and we see that there is a $Y$-relatively compact set $\mathcal K_Y=\{\|\rho\|_{L^2}\leq C\}$ such that $\rho^\tau(t)\in \mathcal K_Y$ for all $t\geq 0$.
Exploiting the above equicontinuity we can apply again the same variant of the Arzel\'a-Ascoli theorem \cite[prop. 3.3.1]{AGS_08} in any bounded time interval to deduce that there exists a subsequence (not relabeled) and $\rho\in \mathcal C([0,T];Y)$ such that $\rho^\tau(t)\to\rho(t)$ in $Y$ for all $t\in [0,T]$.
A further application of Lebesgue's dominated convergence theorem with $\|\rho^\tau(t)\|_Y\leq C$ shows that $\rho^\tau\to\rho$ in $L^p([0,T];Y)$ for all $p\geq 1$ and fixed $T>0$, and by Cantor's procedure
$$
\rho^\tau\to \rho
\qquad \mbox{in }L^p_{\mathrm{loc}}([0,\infty);Y).
$$
Let now $X:=\mathrm{BV}\cap L^\infty(\Omega)\subset\subset L^2(\Omega)=:B$. We just proved that
$$
X\subset\subset B\subset Y
\quad\mbox{and}\quad
\left\{
\begin{array}{l}
 \rho^\tau	\mbox{ is bounded in } L^\infty(0,\infty;X),\\
 \rho^\tau	\mbox{ is relatively compact in } L^p_{\mathrm{loc}}([0,\infty);Y)
\end{array}
\right.
$$
for all $p\geq 1$.
By standard Aubin-Lions-Simon theory \cite[lem. 9]{simon_compact_sets} we get that $\rho^\tau$ is relatively compact in $L^p_{\mathrm{loc}}([0,\infty);B)$ for all $p\geq 1$. In particular we get pointwise a.e. convergence $\rho^\tau(t,x)\to\rho(t,x)$ (up to extraction of a further subsequence), and a last application of Lebesgue's dominated convergence allows to conclude.
The argument is identical for $\tilde{\rho}^\tau$.
\end{proof}
In order to show that the nonlinear terms pass to the limit as in \eqref{eq:assumption_strong_CV} we shall need the following variant of the Banach-Alaoglu theorem with varying measures:
\begin{lem}[compactness for vector-fields]
\label{lem:variant_banach_alaoglu_vector_fields}
Let $\mathcal O\subset \R^m$ be an open set (not necessarily bounded), $\{\sigma_n\}_{n\geq 0}\subset \M^+(\mathcal O)$ a sequence of finite nonnegative Radon measures narrowly converging to $\sigma\in \M^+(\mathcal O)$, and $\v_n$ a sequence of vector fields on $\mathcal O$. If
$$
\|\v_n\|_{L^2(\mathcal O,\rd\sigma_n;\R^m)}\leq C
$$
then there exists $\v\in L^2(\mathcal O,\rd\sigma;\R^m)$ such that, up to extraction of some subsequence,
$$
\forall \,\zeta\in\mathcal C^\infty_c(\mathcal O;\R^m):\qquad \lim\limits_{n\to\infty}\int_{\mathcal O}\v_n\cdot \zeta \rd \sigma_n=\int_{\mathcal O}\v\cdot \zeta \rd \sigma
$$
and
$$
\|\v\|_{L^2(\mathcal O,\rd\sigma;\R^m)}\leq \liminf\limits_{n\to\infty}\|\v_n\|_{L^2(\mathcal O,\rd\sigma_n;\R^m)}.
$$
\end{lem}
\noindent
The proof can be found in \cite[thm. 5.4.4]{AGS_08} for probability measures, see also \cite[prop. 5.3]{KMV_15} for an abstract version.
As anticipated, we have now
\begin{prop}
\label{prop:compactness_check}
Assume \eqref{hyp:structure_U_Psi}. Then $\rho^\tau,\tilde\rho^\tau$ satisfy the compactness assumption \eqref{eq:assumption_strong_CV} in Theorem~\ref{theo:CV_to_weak_solution}. 
\end{prop}
\begin{proof}
From the strong $L^1_{\mathrm{loc}}([0,\infty);L^1)$ convergence in Proposition~\ref{prop:strong_CV_L1_loc_L1} and the uniform $L^1\cap L^\infty(\Omega)$ bounds in Proposition~\ref{prop:propagation_BV}, a straightforward application of Lebesgue's dominated convergence theorem yields strong convergence $ \sqrt{\rho^\tau}\frac{\sqrt{\tilde\rho^\tau}+ \sqrt{\rho^\tau} }{2}  U'(\rho^\tau) \to\rho U'(\rho)$ at least in $L^1_{\mathrm{loc}}((0,\infty)\times\Omega)$. 
Therefore the reaction terms pass to the limit as in \eqref{eq:assumption_strong_CV}, and we only have to check that the diffusion part does too.

Let $\mathbf t^{n+\frac 12}$ be the (backwards) optimal map from $\rho^{n+\frac 12}$ to $\rho^n$, and recall that the Euler-Lagrange equation \eqref{eq:euler_lagrange_wasserstein_riemannian} holds with $\mu=\rho^n$ and minimizer $\rho^*=\rho^{n+\frac 12}$. An easy density argument shows that \eqref{eq:euler_lagrange_wasserstein_riemannian} can in fact be written as $\frac{\operatorname{id}-\mathbf t^{n+\frac 12}}{\tau}=-\nabla U'(\rho^{n+\frac 12})$ in $L^2(\rd\rho^{n+\frac 12})$, which should be interpreted as an equality in the tangent plane $T_{\rho^{n+\frac 12}}\M^+_\W$.
Taking thus the $L^2(\rd\rho^{n+\frac 12})$ norm we obtain
\begin{equation*}
 \tau\|\nabla U'(\rho^{n+\frac 12})\|^2_{L^2(\rd\rho^{n+\frac 12})}=\frac{1}{\tau}\|\operatorname{id}-\mathbf t^{n+\frac 12}\|^2_{L^2(\rd\rho^{n+\frac 12})}=\frac{1}{\tau}\W^2(\rho^{n+\frac 12},\rho^n).
\end{equation*}
Recalling that the interpolated curve $\tilde{\rho}^\tau(t)$ is piecewise constant and summing from $n=0$ to $n=\lceil T/\tau\rceil +1$ for fixed any $T>0$, we obtain from the total square-distance estimate \eqref{eq:total_square_distance_hybrid}
\begin{equation}
\label{eq:energy_estimate_Wasserstein}
\int_0^T\int_\Omega |\nabla U'(\tilde{\rho}^\tau(t))|^2\rd \tilde\rho^\tau(t)\,\rd t\leq C
\quad\Leftrightarrow\quad
\int_{\mathcal O}|\nabla U'(\tilde{\rho}^\tau)|^2 \rd\sigma^\tau\leq C
\end{equation}
with $\mathcal O=(0,T)\times\Omega\subset \R^{1+d}$ and $\rd\sigma^\tau(t,x)=\rd \tilde\rho^\tau_t(x)\otimes\rd t$. Recall that $\|\tilde \rho^\tau(t)\|_{L^1(\Omega)}\leq \|\rho^0\|_{\Omega}$, so that $\sigma^\tau$ is really a \emph{finite} measure on $\mathcal O$ for finite $T>0$.
From the strong $L^1_{\mathrm{loc}}([0,\infty);L^1)$ convergence $\tilde\rho^\tau\to \rho$ (Proposition~\ref{prop:strong_CV_L1_loc_L1}) it is easy to check that $\sigma^\tau$ converges narrowly to $\rd\sigma(t,x)=\rd \rho_t(x)\otimes \rd t=\rho(t,x) \rd x\rd t$.
Applying Lemma~\ref{lem:variant_banach_alaoglu_vector_fields} we see that there is a vector-field $\v\in L^2(\mathcal O,\rd\sigma)=L^2(0,T;L^2(\rd\rho _t))$ such that, up to extraction of a subsequence,
\begin{equation*}
\label{eq:weak_CV_rho_nabl_IU'}
\int_0^T\int_\Omega \tilde\rho^\tau\nabla U'(\tilde\rho^\tau)\cdot \zeta
\to
\int_0^T\int_\Omega \rho(t,x)\v(t,x)\cdot \zeta(t,x) \, \rd x\rd t
\end{equation*}
for all $\zeta\in\mathcal C^\infty_c((0,T)\times\Omega;\R^n)$.
In order to identify the weak limit $\v$, recall that the thermodynamic pressure $P(\rho):=\rho U'(\rho)-U(\rho)$.
Since $P'(\rho)=\rho U''(\rho)$ our assumptions on $U$ show that $P$ is Lipschitz in any bounded interval $\rho\in[0,M]$.
With the strong convergence $\rho^\tau\to \rho$ and the uniform $L^1\cap L^\infty(\Omega)$ bounds one immediately gets $P(\tilde\rho^\tau)\to P(\rho)$ in $L^1_{\mathrm{loc}}((0,\infty)\times\Omega)$, and as a consequence $\nabla P(\tilde\rho^\tau)\rightharpoonup \nabla P(\rho)$ in the sense of distributions $\mathcal D'((0,T)\times\Omega)$.
Note that the measure $\rd\sigma(t,x)=\rd\rho_t(x)\otimes \rd t$ is finite on any subdomain $(0,T)\times\Omega$, hence $\v\in L^2(\mathcal O,\rd\sigma)\subset L^1(\mathcal O,\rd\sigma)$ and $\rho \v\in L^1((0,T)\times\Omega)$. 
Writing $\nabla P(\rho)=P'(\rho)\nabla\rho =\rho U''(\rho)\nabla\rho=\rho\nabla U'(\rho)$ we conclude that $\rho\v=\nabla P(\rho)=\rho\nabla U'(\rho)$, thus $\v=\nabla U'(\rho)$ at least in $L^2(\rd\rho)$.
A further diagonal extraction shows that the limit $\v$ can be chosen independent of $T$, and the proof is complete.
\end{proof}
{ 
As an immediate consequence, we get
\begin{theo}
Assume \eqref{hyp:structure_U_Psi}. Then, up to extraction of a discrete subsequence not relabeled here, the solution of the $\W$-$\FR$ splitting scheme $\rho^\tau$  converges to a weak solution $\rho$ of the PDE \eqref{eq:PDE}.
\end{theo}
\begin{proof}
Simply use Proposition~\ref{prop:compactness_check} to apply Theorem~\ref{theo:CV_to_weak_solution}.
\end{proof}
}
Our next and final result illustrates perhaps even better the deep interplay between our two-steps variational discretization and the full $\d$ metric:
\begin{prop}
 \label{prop:EDI}
In addition to \eqref{hyp:structure_U_Psi}, assume that $\mathcal{F}(\rho)$ is geodesically convex with respect to the $\W$ structure, i-e $\rho P'(\rho)\geq \left(1-\frac{1}{d}\right)P(\rho)$ with $P(\rho)=\rho U'(\rho)-U(\rho)$ \cite{villani_small}.
Then we have
\begin{equation}
\label{eq:EDI}
 \mathcal F(\rho(t_2))+\int_{t_1}^{t_2}\int_\Omega(|\nabla U'(\rho)|^2+|U'(\rho)|^2)\,\rd\rho\,\rd t\leq \mathcal F(\rho(t_1))
\end{equation}
and for all $0\leq t_1\leq t_2$.
\end{prop}
From the discussion in section \ref{section:FRHKW} we known that $\|U'(\rho)\|^2_{H^1(\rd\rho)}$ can be interpreted either as the metric slope $|\partial \mathcal F(\rho)|^2=\|\grad_\d\mathcal F(\rho)\|^2_\KFR$ or, through the continuity equation $\partial_t\rho=\dive(\rho\nabla U'(\rho))-\rho U'(\rho)$, as the metric speed $|\rho'(t)|^2$ with respect to our distance $\d$.
Hence \eqref{eq:EDI} can be rephrased as the Energy Dissipation Inequality (EDI)
$$
\mathcal F(\rho(t_2))+\int_{t_1}^{t_2}\left\{\frac 12 |\rho'(t)|^2 + \frac 12|\partial\mathcal F(\rho(t))|^2\right\}\,\rd t\leq \mathcal F(\rho(t_1)),
$$
which is one of the possible formulations of gradient flows in abstract metric spaces.
{ 
We refer the reader to \cite{ambrosio_gigli_user_guide_OT_13,AGS_08} for the connection between EDIs in abstract metric spaces and gradient flow formulations.
However, and to the best of our knowledge, no full and tractable characterizations of metric speeds $|\rho'(t)|$ and metric slopes $|\partial \mathcal F(\rho)|$ are available at this early stage of the general $\KFR$ theory (see however \cite{KMV_15} for the characterization of Lipschitz curves).
For the sake of rigor we thus prefer to state the dissipation inequality in the PDE-oriented form \eqref{eq:EDI}, rather than in the abstract metric setting.}

Note that \eqref{hyp:structure_U_Psi} already implies $\rho U''(\rho)+U'(\rho)/2\geq 0$, which is equivalent to geodesic convexity with respect to $\H$.
Thus we essentially assumed here that $\mathcal F$ is separately geodesically convex with respect to each of the $\W,\H$ structures, respectively, and it is not surprising that we recover in the end a dissipation inequality for the full $\KFR$ metrics.
\begin{proof}
Let $\mathbf t^{n+\frac 12}$ be the optimal map from $\rho^{n+\frac 12}$ to $\rho^n$. By the above-tangent characterization of the displacement convexity with respect to $\W$ \cite[prop. 5.29]{villani_small} we have
\begin{align*}
\mathcal F(\rho^{n}) & \geq \mathcal F(\rho^{n+\frac 12})+\int_\Omega (\mathbf t^{n+\frac 12}-\operatorname{id})\cdot \nabla U'(\rho^{n+\frac 1 2})\rd\rho^{n+\frac 12}\\
& = \mathcal F(\rho^{n+\frac 12})+\tau\int_\Omega|\nabla U'(\rho^{n+\frac 1 2})|^2\rd\rho^{n+\frac 12},
\end{align*}
where the last equality follows by reinterpreting the Euler-Lagrange \eqref{eq:euler_lagrange_wasserstein_riemannian} as $\mathbf t^{n+\frac 12}-\operatorname{id}=\tau\nabla U'(\rho^{n+\frac 1 2})$ in $L^2(\rd\rho^{n+\frac{1}{2}})$.

For the reaction part let us recall that $\rho U''(\rho)+\frac{U'(\rho)}{2}\geq 0$ corresponds to the convexity of $s\mapsto U(s^2)$ in $s=\sqrt{\rho}$.
Using this convexity we obtain
\begin{align*}
\mathcal F(\rho^{n+\frac 12}) & \geq \mathcal F(\rho^{n+1})+\int_\Omega 2\sqrt{\rho^{n+1}}U'(\rho^{n+1})\,\left(\sqrt{\rho^{n+\frac 12}} - \sqrt{\rho^{n+1}}\right)\\
& = \mathcal F(\rho^{n+1})+\tau \int_\Omega |U'(\rho^{n+1})|^2\rd \rho^{n+1},
\end{align*}
 where the last equality follows now by reinterpreting  the Euler-Lagrange equation \eqref{eq:euler_lagrange_hellinger_general} as $2\frac{\sqrt{\rho^{n+1}}-\sqrt{\rho^{n+\frac 12}}}{\sqrt{\rho^{n+1}}}=-\tau U'(\rho^{n+1})$ in $L^2(\rd\rho^{n+1})$.
We get altogether
$$
\mathcal F(\rho^{n+1})+ \tau\left(\int_{\Omega}|\nabla U'(\rho^{n+\frac 12})|^2\rd\rho^{n+\frac{1}{2}}+\int_\Omega |U'(\rho^{n+1})|^2\rd\rho^{n+1}\right) \leq \mathcal{F}(\rho^n).
$$
For any $0\leq t_1\leq t_2$ let now $N_1,N_2\in \N$ such that $\rho^\tau(t_i)=\rho^{N_i}$, and $T_i=N_i\tau$.
Summing the previous inequality from $n=N_1$ to $n=N_2-1$ gives
\begin{equation}
 \label{eq:EDI_approx}
\mathcal F(\rho^\tau(t_2))
+\int_{T_1}^{T_2}\int_{\Omega}|\nabla U'(\tilde\rho^\tau)|^2\,\rd\tilde\rho^\tau\rd t
+\int_{T_1}^{T_2}\int_\Omega |U'(\rho^\tau)|^2\,\rd\rho^\tau\rd t
\leq \mathcal F(\rho^\tau(t_1)).
\end{equation}
We proved in Proposition~\ref{prop:compactness_check} that $\tilde\rho^\tau\nabla U'(\tilde{\rho}^\tau)\rightharpoonup 
\rho \nabla U'(\rho)$, and observe that $T_i\to t_i$ as $\tau\to 0$. From the energy estimate \eqref{eq:energy_estimate_Wasserstein} and the lower semi-continuity in Lemma~\ref{lem:variant_banach_alaoglu_vector_fields} we deduce that
\begin{equation*}
\int_{t_1}^{t_2}\int_{\Omega}|\nabla U'(\rho)|^2\,\rd\rho\,\rd t
\leq \liminf\limits_{\tau\to 0}\int_{T_1}^{T_2}\int_{\Omega}|\nabla U'(\tilde\rho^\tau)|^2\,\rd\tilde\rho^\tau\rd t,
\end{equation*}
and from the strong convergence in Proposition~\ref{prop:strong_CV_L1_loc_L1} with the uniform $L^1\cap L^\infty(\Omega)$ bounds (Proposition~\ref{prop:propagation_BV}) it is easy to see that
$$
\int_{t_1}^{t_2}\int_{\Omega}|U'(\rho)|^2\,\rd\rho\,\rd t =\lim\limits_{\tau\to 0}\int_{T_1}^{T_2}\int_\Omega |U'(\rho^\tau)|^2\,\rd\rho^\tau\rd t.
$$
Similarly one can verify that
$$
\forall\,t\geq 0:\qquad \mathcal{F}(\rho^\tau(t))=\int_\Omega U(\rho^\tau(t))\to \int_\Omega U(\rho(t)) =\mathcal F(\rho(t)).
$$
Indeed with our assumptions $U$ is Lipschitz in any bounded interval $\rho\in [0,M]$, $\|\rho^\tau(t)\|_{L^\infty}\leq M=\|\rho^0\|_{L^\infty}$ uniformly in $\tau$, and in the first proof of Proposition~\ref{prop:strong_CV_L1_loc_L1} we obtained strong $L^1(\Omega)$ convergence $\rho^\tau(t)\to \rho(t)$ pointwise in time.
As a consequence we can pass to the $\liminf$ in \eqref{eq:EDI_approx} to retrieve \eqref{eq:EDI} and the proof is complete.
\end{proof}

\subsection*{Acknowledgments}
LM was partially supported by the Portuguese National Science Foundation through fellowship BPD/88207/2012 and by the UT Austin/Portugal CoLab program \emph{Phase Transitions and Free Boundary Problems}. T. O. Gallou\"et was supported by the ANR project ISOTACE (ANR-12-MONU-013) hosted at CMLS, \'Ecole polytechnique, CNRS, Universit\'e Paris-Saclay and by the fond de la Recherche Scientifique-FNRS under grant MIS F.4539.16.
 We whish to thank the anonymous referees for their useful comments and suggestions.

%
\bibliographystyle{plain}
\bibliography{./biblio}

\end{document}